\documentclass[reqno]{amsart}

\usepackage[letterpaper, hmarginratio=1:1]{geometry}
\usepackage[hidelinks]{hyperref}                     
\usepackage{bbm,bm}
\usepackage{stmaryrd}
\usepackage{enumerate}
\usepackage{amssymb}
\usepackage{amsmath, amsthm}
\usepackage{cleveref}
\usepackage[mathscr]{euscript}
\usepackage{tikz}
\usetikzlibrary{calc}

\Crefname{appendix}{Appendix}{Appendices}

\usepackage[
backend=biber,
style=alphabetic,
giveninits=true,
url = false,
doi = false,
isbn =false,
maxbibnames=99,
maxalphanames=99,
]{biblatex}
\renewbibmacro{in:}{}
\DeclareFieldFormat{title}{#1} 
\DeclareFieldFormat[article]{title}{#1}
\DeclareFieldFormat[inbook]{title}{#1}
\DeclareFieldFormat[incollection]{title}{#1}
\AtBeginBibliography{\small}

\newtheorem{theorem}{Theorem}[section]

\newtheorem{lemma}[theorem]{Lemma}

\newtheorem{proposition}[theorem]{Proposition}

\theoremstyle{definition}
\newtheorem{definition}[theorem]{Definition}
\newtheorem{remark}[theorem]{Remark}

\newcommand\C{\mathbb{C}}

\newcommand\E{\mathbb{E}}

\newcommand\N{\mathbb{N}}

\renewcommand\P{\mathbb{P}}
\newcommand\R{\mathbb{R}}

\newcommand\bbH{\mathbb{H}}

\newcommand\bbT{\mathbb{T}}

\newcommand{\bbV}{\mathbb{V}}

\newcommand{\rank}{\mathrm{rank}}
\newcommand{\Law}{\mathrm{Law}}
\newcommand{\diag}{\mathrm{diag}}

\newcommand\be{\begin{equation}}
\newcommand\ee{\end{equation}}

\newcommand\Var{\operatorname{Var}}
\newcommand{\tr}{\operatorname{Tr}}
\newcommand{\eps}{\varepsilon}
\renewcommand{\epsilon}{\varepsilon}
\newcommand{\iu}{\mathrm{i}}

\renewcommand{\hat}{\widehat}
\renewcommand{\log}{\ln}

\renewcommand{\phi}{\varphi}

\def\beq{\begin{equation}}
\def\eeq{\end{equation}}
\newcommand{\dist}{\operatorname{dist}}

\numberwithin{equation}{section}
\renewcommand{\L}{\mathfrak L}

\newcommand{\bex}{\begin{equation*}}
\newcommand{\eex}{\end{equation*}}

\renewcommand{\Re}{\operatorname{Re}}
\renewcommand{\Im}{\operatorname{Im}}

\newcommand{\one}{\mathbbm{1}}

\title{Mobility Edge for the Anderson Model on Random Regular Graphs}
\author{Suhan Liu}
\author{Patrick Lopatto}
\date{\today}
\begin{document}
\begin{abstract}
We determine the phase diagram of the Anderson tight-binding model on random regular graphs with Gaussian disorder and sufficiently large degree. In particular, we prove that if the degree is fixed and the number of vertices goes to infinity, the spectrum asymptotically consists of a finite delocalized interval surrounded by two unbounded localized components. 
Our argument uses a recent description of the spectrum of the tight-binding model on the Bethe lattice (Aggarwal--Lopatto, 2025). By viewing the Bethe lattice as the local limit of a random regular graph, and establishing suitable concentration, eigenvalue-counting, and resolvent estimates, we transfer this characterization of the spectrum of the limiting model to the finite-volume setting. 
\end{abstract}
\maketitle

\setcounter{tocdepth}{1}
\tableofcontents

\section{Introduction}

\subsection{Background}
In a disordered medium, such as a doped semiconductor, the diffusion of waves ceases completely when the disorder is sufficiently strong. This principle, known as \emph{Anderson localization}, has played a fundamental role in the development of solid state physics \cite{abrahams201050}. Significant effort has been devoted to studying Anderson localization from first principles, and a detailed picture based on the methods of theoretical physics has emerged, augmented by a substantial body of numerical evidence. However, our mathematically rigorous understanding remains highly incomplete.

A standard setting for studying wave transport in disordered media is the tight-binding model, which served as the basis for Anderson's original investigations \cite{Anderson1958absence}. 
Let $\mathcal G$ be a graph (finite or infinite), let $A$ denote the adjacency operator of $\mathcal G$, and  let $V$ denote a diagonal operator of independent, identically distributed (i.i.d.) random variables. A particle whose time evolution is governed  by $A$ moves freely by nearest-neighbor hopping. In contrast, the time evolution induced by $V$ traps a particle at its initial site. 
The tight-binding Hamiltonian is defined as $H_t =  -t A  +V$; it represents a contest between the opposite influences of $A$ and $V$ mediated by $t$, the inverse disorder strength. When $t$ is sufficiently small,  $V$ dominates, and the particle remains localized near its initial position. However, when $t$ is sufficiently large, one expects that extended states may emerge, at least for certain energies and graph geometries, leading to a region (or regions) of delocalization in the spectrum separated from regions of localization by sharp transitions called \emph{mobility edges} \cite{aizenman2015random}.

An important special case is when $\mathcal G$ is the infinite $d$-regular tree, also known as the Bethe lattice. 
While simpler than more physically realistic graphs, such as $\mathbb{Z}^3$, its lack of cycles renders it more amenable to analysis  \cite{abou1973selfconsistent,abou1974self,mirlin1991localization,miller1994weak}. Because of this, it has served as a distinguished testbed for theoretical investigations. For instance, the self-consistent theory of Anderson localization was developed using equations which are approximate on a true lattice, but become exact on the Bethe lattice \cite{abou1973selfconsistent,abou1974self}. Much is now known rigorously about this model, including delocalization at small energies \cite{klein1994absolutely, klein1998extended,klein1996spreading,aizenman2006stability,froese2007absolutely}, localization at large energies \cite{aizenman1994localization,aizenman1993localization}, and a nearly complete description of the phase diagram at low disorder  (omitting only small intervals around the expected locations of the mobility edges) \cite{drogin2025regular}. Further, \cite{aizenman2013resonant} established a criterion for absolutely continuous spectrum based on the fractional moments of resolvent entries. This was used in \cite{bapst2014large}, along with the Simon--Wolff criterion for localization \cite{SCSRP,aizenman1993localization}, to provide quantitative bounds on the endpoints of the regions of localization and delocalization, mapping a large part of the spectrum. As the degree of the Bethe lattice tends to infinity, these bounds become increasingly better and characterize the entire spectrum in the limit (but not for any finite degree). Recently, building on the work of \cite{bapst2014large,SCSRP,aizenman1993localization,aizenman2013resonant,aggarwal2022mobility}, mobility edges separating regions of delocalization and localization were shown to exist in \cite{aggarwal2025mobility} when the degree $d$ is sufficiently large and the potential has unbounded support.

A close cousin of the Bethe lattice is the random regular graph, which is generally understood to be its correct finite-size analog. (Simply truncating the Bethe lattice to a finite tree leads to a graph where most vertices are leaves, which produce undesirable boundary effects.) Since random regular graphs converge locally to the Bethe lattice, their spectral and eigenfunction observables measure how Bethe-lattice physics is modified by long loops and finite-size effects, providing a controlled high-dimensional reference for comparison with finite-dimensional Anderson localization \cite{biroli2022critical,biroli2010anderson,tarquini2017critical}. Further, there has been a resurgence of interest in the tight-binding model on the random regular graph due to its connection with many-body localization, a phenomenon that has recently attracted significant interest in the physics community \cite{pascazio2023anderson}. Specifically, in suitable many-body problems, certain hierarchical models in Fock space can be approximately mapped to tight-binding models on locally tree-like graphs, for which a random regular graph is a useful proxy \cite{tikhonov2016anderson}.

There have been only a few rigorous results for the tight-binding model on random regular graphs. 
In \cite{geisinger2013convergence}, it was shown that for compactly supported and bounded disorder densities, the integrated density of states converges to the integrated density of states of the Bethe lattice, as $d$ is fixed and the number of vertices $N$ tends to infinity, with a rate of convergence of order $1/\log N$. This work also showed that if the eigenvectors in a finite energy window in the absolutely continuous spectrum for the Bethe lattice are averaged, weighted by their squared amplitudes at a uniformly random vertex $x_0$, the resulting weighted mass does not concentrate in small balls around $x_0$; this is a form of eigenvector delocalization.
Further, \cite{geisinger2015poisson} showed a criterion for the presence of Poisson spectral statistics based on the fractional moment decay of off-diagonal resolvent entries and applied it to the edges of the spectrum (for compactly supported disorder distributions). Finally, it was shown in \cite{anantharaman2017quantum} that in the weak-disorder regime, for energy windows in the  absolutely continuous spectrum for the Bethe lattice, most eigenfunctions are asymptotically equidistributed on their coordinates (again for compactly supported disorder distributions).

In this article, we are interested in showing the existence of mobility edges that sharply separate  regions of delocalization and localization in random regular graphs. Our main result establishes this for sufficiently large (but fixed) degrees $d$ and a wide range of disorder strengths when the disorder distribution is Gaussian. We choose to concentrate on this setting for brevity, and because it already illustrates the phenomenology predicted by the physics literature. However, our methods are rather flexible, and we comment on possible generalizations below, after stating our main result. 

\subsection{Main Result}\label{s:model}

Fix an integer $d\ge 3$, and let $\N$ denote the positive integers. 
When discussing $d$-regular random graphs with $N$ vertices, we always restrict to the subset of integers $N$ such that $N \ge d+1$ and $Nd$ is even (so that the configuration model construction of $d$-regular graphs is valid). Statements quantifying over $N$ implicitly quantify over only this subset, and all limits $N \rightarrow \infty$ are implicitly along the subsequence of these values. 

Let $\mathcal G_N$ be a (uniformly) random $d$-regular graph on the vertex set
$[N]=\{1,\dots,N\}$, and let $A_N$ denote its adjacency matrix.  
Let $(g_i)_{i=1}^\infty$ be i.i.d.\ Gaussian random variables with mean zero and variance one (independent of $\mathcal G_N$).  Define the diagonal matrix 
\[
V_N = \diag(g_1,\dots,g_N).
\]
For every $t>0$, we define the symmetric random matrix 
\begin{equation}\label{eq:Hn-def}
W_{N,t} = -t A_N + V_N.
\end{equation}
We will often abbreviate $W_{N,t} = W_N$, as the value of $t$ will never depend on $N$. 

Let $(\lambda_k,u_k)_{k=1}^N$ be an orthonormal eigendecomposition of $W_N$, so that 
\[
W_N u_k=\lambda_k u_k,\qquad \langle u_k,u_\ell\rangle=\delta_{k\ell},
\]
where $\delta_{k\ell}$ is the Kronecker delta function. The eigenvectors $u_k$ are generically determined only up to a sign, which we may select arbitrarily. 
We let $\Lambda$ denote the multiset of eigenvalues of $W_N$ (which may have duplicated elements due to repeated eigenvalues), and given an interval $I\subset \R$, we let $\Lambda_I$ be the multiset of eigenvalues contained in $I$. While these objects depend on $d$, $t$, and $N$, we suppress this dependence in the notation.

To study (de)localization of the eigenvectors $u_k$, we follow \cite{bordenave2013localization} and consider averages of (powers of) coordinates of eigenvectors associated with eigenvalues in small intervals of the real line. 

\begin{definition}\label[definition]{def:PI_QI_anderson}
Fix a compact interval $I\subset\mathbb R$ and index  $j\in[N]$. We define 
\[
P_I(j) = \frac{N}{|\Lambda_I|}\sum_{\lambda_i\in \Lambda_I} u_i(j)^2, \qquad 
Q_I=\frac{1}{N}\sum_{j=1}^N P_I(j)^2.
\]
When $|\Lambda_I|=0$, we set $P_I(j) = 0$. 
\end{definition}
To understand the meaning of these observables, we consider two extreme cases. First, suppose that all eigenvectors with eigenvalues in $I$ are as delocalized as possible, so that $u_i(j)^2 = N^{-1}$ for all $j \in [N]$ and $i$ such that $\lambda_i \in I$. Then $P_I(j) = 1$ for all $j$, implying $Q_I = 1$. In particular, $Q_I$ stays bounded as $N$ grows. Next, consider the case where $u_i(j) = \delta_{ij}$ for all $i,j \in [N]$, in which all eigenvectors are maximally localized. Then $P_I(j)=0$ if $\lambda_j \notin I$ and $N | \Lambda_I |^{-1}$ if $\lambda_j \in I$, implying $Q_I = N | \Lambda_I |^{-1}$. If $| \Lambda_I | \le C |I | N$ for all $N$, where $C>0$ is a constant that does not depend on $I$ or $N$, and $|I|$ denotes the length of $I$, then $Q_I$ becomes unbounded as $|I|$ tends to zero. 

Therefore, it is reasonable to distinguish between delocalization and localization based on the behavior of $Q_I$ as $I$ shrinks (assuming the relevant upper bound on $| \Lambda_I |$, which will be shown in \Cref{l:spectral_density_bounds}  below). This observation motivates the next definition. 

\begin{definition}\label{def:deloc_loc_anderson}
Fix an energy $E\in\mathbb R$.
We say that \emph{delocalization holds at $E$} for $W_{N}$ if there exist
constants $C(E)>1$ and $c(E)>0$ such that for all $\varepsilon\in(0,c )$,
\[
\lim_{N\to\infty}\mathbb P\Big( Q_{I(\varepsilon)}\le C \Big)=1,
\]
where $I(\varepsilon)=[E-\varepsilon,\,E+\varepsilon]$. 
We say that \emph{localization holds at $E$} if for every $D>0$,  there exists
$c(D)>0$ such that for all $\varepsilon\in(0,c)$,
\[
\lim_{N\to\infty}\mathbb P\Big( Q_{I(\varepsilon)}\ge D\Big)=1.
\]
\end{definition}
We now state our main theorem. The scaling of the inverse disorder strength as $(d \ln d)^{-1}$ is conventional; it ensures the mobility edge locations tend to constant-order limits as $d$ grows large. The constants arise from the fact that the supremum of the density of a mean zero, variance one Gaussian random variable is $(2 \pi )^{-1/2}$.

\begin{theorem}\label{t:main}
For every real number $\L >0$, there exists a constant $d_0(\L) > 0$ such that the following holds for all $d \ge d_0$. Fix $g \in \R$ such that 
\[
\sqrt{\frac{\pi}{8}} + \frac{1}{\L} < g < \L,
\]
and let 
\[ t= g ( d \ln d )^{-1},
\qquad
\mathfrak E = \sqrt{2 \ln\!\left( \frac{4g}{\sqrt{2\pi}} \right)} .
\]
Then there exists $\mathfrak M(d,g) > 0$ such that $|\mathfrak M  - \mathfrak E| < \L^{-1}$, localization holds for $W_{N}$ at all values in the set $\{E \in \R : |E| > \mathfrak M\}$, and delocalization holds for $W_{N}$ at all values in the set $\{E \in \R : |E| < \mathfrak M\}$.
\end{theorem}

\subsection{Proof Ideas} 

As mentioned previously, the spectral decomposition of the tight-binding model on the Bethe lattice was characterized  in \cite{aggarwal2025mobility} (for the disorder regime considered in our main result). Since the Bethe lattice is the local limit of random regular graphs, it is natural to attempt to use the Bethe lattice result to prove \Cref{t:main}, and this is indeed our strategy. However, this route requires suitable estimates for the tight-binding model on the regular random graph, which we enumerate in \Cref{s:criteria}. Specifically, we require concentration for fractional moments of the diagonal resolvent entries (\Cref{l:concentration}), upper and lower bounds on the eigenvalue density (\Cref{l:spectral_density_bounds}), convergence of the diagonal resolvent entries of the random regular graph to the diagonal resolvent entry at the root of the Bethe lattice (\Cref{thm:diag-resolvent-fixed-z}), and, in the delocalized regime,  a bound on the second moment of the imaginary part of the resolvent at the root of the Bethe lattice that is uniform in the imaginary part of the spectral parameter (\Cref{l:self-energy}). Of these, the last is perhaps the most technically challenging. We prove it by first establishing a Cauchy-type tail bound on the imaginary part of the relevant resolvent entry, then bootstrapping that bound to an improved polynomial tail decay bound. 

Generalizations of our result are possible in at least two directions. First, we have chosen to concentrate on the case of Gaussian disorder to streamline the presentation. The spectral decomposition of the tight-binding model on  the Bethe lattice was characterized in \cite{aggarwal2025mobility} for a wide class of unbounded potentials; in the general case, it is slightly more complicated to describe, with interleaved regions of localization and delocalization determined by the disorder density. The techniques we present suffice to establish the analogous result for random regular graphs with corresponding potentials under the additional assumption that the density decays sufficiently rapidly (for example, exponentially), so that the concentration argument for \Cref{l:concentration} below, based on Gaussian concentration for Lipschitz functions, can be replaced by one based on truncation and Talagrand's concentration inequality for bounded distributions.

Second, the requirement that $d$ be very large in \Cref{t:main} comes primarily from the analogous Bethe lattice result in \cite{aggarwal2025mobility}, not our arguments that show its conclusions transfer to random regular graphs. The main input we require from the limiting model is an understanding of how the imaginary part of the resolvent at the root behaves in the upper half plane as the imaginary part of its spectral parameter goes to zero (see \Cref{lem:tree_resolvent_criterion_anderson}). In particular, it seems likely that our methods suffice to transfer the nearly complete description of the phase diagram for the Bethe lattice at low disorder recently established in \cite{drogin2025regular} to random regular graphs. 

\subsection{Acknowledgments} 

P.\ L.\ thanks A.\ Aggarwal for several helpful conversations. Both authors were partially supported by NSF grant DMS-2450004.

\section{Proof Outline}\label{s:proof}
\subsection{Notations and Conventions}
\label{s:bethelattice}
We set $\bbH = \{ z \in \C : \Im z > 0 \}$. For $z \in \bbH$, we typically use the notation $z = E + \iu \eta$, where $E = \Re z$ and $\eta = \Im z$. 
For all $z \in \bbH$, we define the resolvent of $W_N$ by 
\[
G_N(z) = (W_N - z)^{-1},
\]
and we frequently abbreviate $G=G_N$.

Given an integer $d \ge 3$, we let $\bbT = \bbT_d$ be a $d$-regular rooted tree graph. We denote the root by $0$, and the vertex set of $\bbT$ by $\bbV = \bbV^{(d)}$. For all $v \in \bbV$, we let $\delta_v$ denote the unit vector supported at $v$, meaning that $\delta_v(w) = \delta_{vw}$ for all $w \in \bbV$. 

We define the adjacency operator $A = A^{(d)}$ for $\bbT$ by setting $\langle \delta_v, A \delta_w \rangle = 1$ if there is an edge between $v$ and $w$ in $\bbT$ and setting $\langle \delta_v, A \delta_w \rangle = 0$ otherwise. It extends by linearity to a bounded operator $\ell^2(\bbV)$. Further, we let $(g_v)_{v \in \bbV}$ be a collection of i.i.d.\ Gaussian random variables with mean zero and variance one. We define an unbounded operator $V= V^{(d)}$ on $\ell^2(\bbV)$ by setting $\langle \delta_v, V \delta_w \rangle = g_v$ if $v=w$ and $\langle \delta_v, V \delta_w \rangle = 0$ otherwise, then extending by linearity. Then, for every $t >0$, we define the operator 
\[
H_{d,t} = - t A^{(d)} + V^{(d)}.
\]
It is well-known that $H_{d,t}$ is unbounded and self-adjoint on the domain of $V^{(d)}$; see, e.g., \cite{kirsch2008invitation}. 
For all $z \in \bbH$, we let $R(z)$ denote the resolvent of $H_{d,t}$.

\subsection{Mobility Edge on the Bethe Lattice}

The following theorem is a specialization of \cite[Theorem 1.3]{aggarwal2025mobility} to the case of the Gaussian disorder. 

\begin{theorem}\label{t:main2}
For every real number $\L >0$, there exists a constant $d_0(\L) > 0$ such that the following holds for all $d \ge d_0$. Fix $g \in \R$ such that 
\begin{equation}\label{e:g_conditions}
\sqrt{\frac{\pi}{8}} + \frac{1}{\L} < g < \L,
\end{equation}
and let 
\[ t= g ( d \ln d )^{-1},
\qquad
\mathfrak E = \sqrt{2 \ln\!\left( \frac{4g}{\sqrt{2\pi}} \right)} .
\]
Then there exists $\mathfrak M(d,g) > 0$ such that $|\mathfrak M  - \mathfrak E| < \L^{-1}$ and the following claims hold.
\begin{enumerate}
\item For all $E \in \R$ such that $|E| > \mathfrak M$,  we have $ \lim_{\eta \rightarrow 0} \Im R_{00}(E + \iu \eta) =0$ in probability. 
\item 
For all $E \in \R$ such that $|E| < \mathfrak M$, there exists $c(E) > 0$ such that \[\liminf_{\eta \rightarrow 0} \P ( \Im R_{00}(E + \iu \eta) >c ) >c.\]
\end{enumerate}
\end{theorem}

\begin{remark}
In \cite{aggarwal2025mobility}, the conclusion of the theorem is stated in a slightly different way, involving absolutely continuous and pure point spectrum. However, it is readily verified by inspecting the proof that the claims in \Cref{t:main2} are proved as intermediate steps (see, e.g.,  \cite[Lemma 2.7]{aggarwal2025mobility}). We also note that the proof in \cite{aggarwal2025mobility} is for a tree whose root degree is one less than each other vertex. However, as mentioned in that reference, it is straightforward to see that the same argument gives the conclusion for $\bbT_d$. 
\end{remark}

\subsection{Resolvent Criteria for Localization and Delocalization}

The following lemma is the primary novelty of this paper.  It shows that the conclusions of \Cref{t:main2} regarding the resolvent of the Anderson model at the root of the Bethe lattice imply the corresponding (de)localization behavior for the  matrices $W_N$. The proof is given in \Cref{s:criteria}.
\begin{lemma}\label[lemma]{lem:tree_resolvent_criterion_anderson}
Fix $E\in\mathbb R$, $d \in \N$ such that $d\ge 10$, and $t>0$. 
\begin{enumerate}
\item If
$ \lim_{\eta \rightarrow 0} \Im R_{00}(E + \iu \eta) =0$ in probability, then localization holds at $E$ for $W_N$. 
\item If there exists $c>0$ such that $\liminf_{\eta \rightarrow 0} \P ( \Im R_{00}(E + \iu \eta) >c ) >c $, then delocalization holds at $E$ for $W_N$.
\end{enumerate}
\end{lemma}

\subsection{Proof of Main Result}
We are now able to prove our main result.
\begin{proof}[Proof of \Cref{t:main}]
This is an immediate consequence of \Cref{t:main2} and \Cref{lem:tree_resolvent_criterion_anderson}.
\end{proof}

\section{Localization and Delocalization Criteria}\label{s:criteria}
This section is devoted to the proof of \Cref{lem:tree_resolvent_criterion_anderson}. 
\subsection{Resolvent Estimates}

The following lemma collects some basic facts about $G$, which follow readily from the definition of $G$ and the spectral theorem. Proofs may be found in, e.g., \cite{BenaychGeorgesKnowles2018}.

\begin{lemma}
Fix $z = E + \iu \eta \in \bbH$.
\begin{enumerate}
\item For all $i,j \in [N]$, we have the deterministic bound 
\begin{equation}\label{e:trivial}
|G_{ij} (z) | \le \eta^{-1}.
\end{equation}
\item For every $i \in [N]$, we have 
\begin{equation}\label{e:ward}
\sum_{j=1}^N \big|G_{ij}(z) \big|^2 = \frac{ \Im G_{ii} }{\eta}.
\end{equation}
\end{enumerate}

\end{lemma}

We also require the next lemma, on the uniform integrability of $\Im R_{00}$.
\begin{lemma}\label[lemma]{l:ui}
Fix $E \in \R$, $t>0$, and $r \in(0,1)$. The set of random variables $\{ (\Im R_{00}(E + \iu \eta ) )^r \}_{ \eta \in (0,1)}$ is uniformly integrable.
\end{lemma}
\begin{proof}
Set $\alpha=(1+r)/2$. 
By \cite[Lemma 3.3]{aggarwal2025mobility}, applied with $L=0$ and exponent $\alpha$,
\[
\sup_{\eta\in(0,1)}
\E\!\left[\left(\big(\Im R_{00}(E+\iu\eta)\big)^r\right)^{\alpha/r}\right]
\le
\sup_{\eta\in(0,1)} \E\bigl[|R_{00}(E+\iu\eta)|^\alpha\bigr]
<\infty.
\]
Because $\alpha/r>1$, the claim follows from de la Vallée--Poussin's criterion.
\end{proof}
\subsection{Preliminary Lemmas}
To prepare for the proof of \Cref{lem:tree_resolvent_criterion_anderson}, four main ingredients are needed. 
First, we state a concentration inequality for the diagonal resolvent entries, which is shown in \Cref{s:concentration}. 
\begin{lemma}\label[lemma]{l:concentration}
Fix $s \in (0,  \infty)$, $t>0$, and $d\ge 3$. There exists $C(d,s,t) > 1$ such that for all $E\in \R$ and $\eta \ge (\log N)^{-1}$, 
\begin{equation}\label{eq:concentration-fractional}
\mathbb P\!\left(
\left|\frac{1}{N}\sum_{j=1}^N \big (\Im G_{jj}(z) \big )^{s}-\mathbb E\Big[\big (\Im G_{11}(z)\big)^{s}\Big] \right|
\le \eta
\right)\ge 1-C N^{-100}.
\end{equation}
\end{lemma}

The next lemma shows that the diagonal resolvent entries of the Anderson model on the $d$-regular graph converge to the resolvent entry at the root of the $d$-regular tree.
The proof is postponed to \Cref{s:diagonal_convergence}. 
\begin{lemma}
\label[lemma]{thm:diag-resolvent-fixed-z}
Fix $d \ge 3 $, $t>0$, and $z\in \bbH$. 
There exists a coupling of $W_{N,t}$ and $H_{d,t}$  such that
\begin{equation}\label{limit: diagonal convergence}
\lim_{N\rightarrow \infty} G_{11}(z) = R_{00}(z)
\end{equation}
in probability.
\end{lemma}

The following lemma  provides upper and lower bounds on the eigenvalue density of $W_N$. 
It is proved in \Cref{sect: Global Law}.

\begin{lemma}\label[lemma]{l:spectral_density_bounds}
For all $t,M >0$ and $d \ge 10$, there exists $c(d,t,M)>0$ such that the following holds. Fix $E\in [-M, M]$ and $\eta \in (0,1)$,  and set  $I = [E - \eta, E + \eta]$. 
Then 
\begin{equation}\label{eq:counting-lower}
\lim_{N\to\infty}\mathbb P\!\left(|\Lambda_I|\ge c N\eta\right)=1,
\qquad 
\lim_{N\to\infty}\mathbb P\!\left(|\Lambda_I|\le c^{-1} N\eta\right)=1.
\end{equation}
\end{lemma}

Finally, we state a bound on the second moment of $\Im R_{00}(z)$ that is uniform in $\eta$. The proof is given in \Cref{s:negativemoment}. The condition $d \ge 10$ could likely be improved, but we refrain from pursuing this to avoid unnecessary complications.

\begin{lemma}\label[lemma]{l:self-energy}
 Fix $E \in \R$, $d \ge 10$, and $t>0$. Suppose that there exists $c>0$ such that \[ \liminf_{\eta \rightarrow 0} \P ( \Im R_{00}(E + \iu \eta) >c ) >c.\]
Then
\begin{equation}\label{eq:EBinv}
\sup_{\eta\in (0,1]}\E\Big[\big( \Im R_{00}(E + \iu \eta) \big)^{2} \Big]<  \infty.
\end{equation}
\end{lemma}

\subsection{Localization}\label{sect:loc}

For all $s>0$, we define 
\begin{equation}\label{eq:QI-def}
Q_I(s)=\frac{1}{N}\sum_{j=1}^N P_I(j)^s.
\end{equation}
Recalling \Cref{def:PI_QI_anderson}, we have  $Q_I=Q_I(2)$.

\begin{lemma}
\label[lemma]{lem:anderson-loc-QI}
Fix $s,\delta\in(0,1)$ and $E\in\mathbb R$, and suppose that 
\[
\lim_{\eta \rightarrow 0} \Im R_{00} (E + \iu \eta) = 0
\]
in probability. 
Let $I(\eta)=[E-\eta,E+\eta]$. Then there exists $\eta_0(s,\delta,E)>0$ such that for all $\eta\in(0,\eta_0)$,
\[
\lim_{N\to\infty}\mathbb P\!\left(Q_I\!\left(\frac{s}{2}\right)\le \delta\right)=1.
\]
\end{lemma}

\begin{proof}
We set $z=E+\iu\eta$, where $\eta$ is a parameter that may vary throughout the proof (but never depends on $N$). By the spectral theorem, we have for all $\eta >0$ that 
\begin{equation}\label{eq:spectral-Gij}
\Im G_{jj}(z)=\sum_{k=1}^N \frac{\eta\,|u_k(j)|^2}{(\lambda_k-E)^2+\eta^2}.
\end{equation}
For $\lambda_k\in I(\eta)$ we have $(\lambda_k-E)^2+\eta^2\le 2\eta^2$, hence
\[
\sum_{\lambda_k\in\Lambda_I}\big |u_k(j)\big |^2
\le \sum_{k=1}^N \frac{2\eta^2\big |u_k(j)\big |^2}{(\lambda_k-E)^2+\eta^2}
=2\eta \Im G_{jj}(z).
\]
By the first bound in \Cref{l:spectral_density_bounds}, we may restrict to the set where $| \Lambda_I | \neq 0$. 
Multiplying by $N/|\Lambda_I|$ gives
\begin{equation}\label{e:int_step}
P_I(j)\le \frac{2 N \eta}{|\Lambda_I|}\,\Im G_{jj}(z),
\end{equation}
and therefore
\begin{equation}\label{eq:QI-upper-via-ImG}
Q_I\!\left(\frac{s}{2}\right)
=\frac{1}{N}\sum_{j=1}^N P_I(j)^{s/2}
\le \left(\frac{2N\eta}{|\Lambda_I|}\right)^{s/2}\cdot
\frac{1}{N}\sum_{j=1}^N \bigl(\Im G_{jj}(z)\bigr)^{s/2}.
\end{equation}
Combining \eqref{eq:concentration-fractional}, \eqref{eq:counting-lower}, and \eqref{eq:QI-upper-via-ImG} yields a constant $c_1(E) >0$ such that 
\begin{equation} \label{limit: QI}
   \lim_{N\to \infty}\P \left( Q_I\!\left(\frac{s}{2}\right)
\le c_1^{-s/2}\left(\mathbb E \Big[\big (\Im G_{11}(z)\big )^{s/2} \Big]+\eta\right)\right)=1. 
\end{equation}

Next, recalling \eqref{limit: diagonal convergence} and using $|G_{11}(z)|\le 1/\eta$ (from \eqref{e:trivial}), uniform integrability yields
\begin{equation}
    \lim_{N\to \infty}\E\Big[\big(\Im G_{11}(z) \big)^{s/2}\Big]=\E\Big[\big(\Im R_{00}(z) \big)^{s/2}\Big].
\end{equation}
By \Cref{l:ui}, the set of random variables $\{ (\Im R_{00}(E+\iu\eta))^{s/2} \}_{\eta\in (0,1)}$ is uniformly integrable. 
Then the  assumption that $\lim_{\eta \rightarrow 0} \Im R_{00}(E+\iu\eta) = 0$ in probability  yields a constant $c_2(s,\delta, E)>0$ such that for all $\eta\in (0,c_2)$, we have 
\[ \E\Big[ \big (\Im R_{00}(z) \big)^{s/2}\Big]< \frac{c_1^{s/2}\delta}{8}.\] This implies that for sufficiently large $N$, 
\[
\E\Big [ \big (\Im G_{11}(z) \big)^{s/2} \Big]\le \frac{c_1^{s/2}\delta}{4}.
\]
Combining this with \eqref{limit: QI} and taking  $\eta_0 < \min\{c_2,\, c_1^{s/2}\delta/2\}$ completes the proof.
\end{proof}

\begin{proposition}\label[proposition]{prop:anderson-QI-blowup}
Under the hypotheses of \Cref{lem:anderson-loc-QI}, for every $D>0$ there exists $\eta_0(D,E)>0$ such that
for all $\eta\in(0,\eta_0)$,
\[
\lim_{N\to\infty}\mathbb P\!\big(Q_I \ge D\big)=1.
\]
\end{proposition}

\begin{proof}
By the first bound in \Cref{l:spectral_density_bounds}, we may restrict to the set where $| \Lambda_I | \neq 0$. 
Let $\varepsilon\in(0,1)$ and  $p,q>1$ be parameters such that $1/p+1/q=1$. By H\"{o}lder's inequality,
\begin{equation}\label{e:p_power}
N=\sum_{j=1}^N P_I(j)
=\frac{1}{N}\sum_{j=1}^N (NP_I(j))^\varepsilon (NP_I(j))^{1-\varepsilon}
\le \Bigl(N^{\varepsilon p} Q_I(\varepsilon p)\Bigr)^{1/p}\,
\Bigl(N^{(1-\varepsilon)q} Q_I((1-\varepsilon)q)\Bigr)^{1/q}.
\end{equation}
With $s \in (0,1)$, we fix 
\[
\eps = \frac{s}{4-s}, \qquad p = 2 - \frac{s}{2}.
\]
Then $\varepsilon p=s/2$,  $(1-\varepsilon)q=2$, and $(p/q)=1-s/2$. Raising both sides of \eqref{e:p_power} to the $p$th power yields
\[
1\le Q_I(s/2)\cdot Q_I(2)^{\,1-s/2}=Q_I(s/2)\cdot Q_I^{\,1-s/2}.
\]
Thus, if $Q_I(s/2)\le \delta$ then $Q_I\ge \delta^{-1/(1-s/2)}$, and the claim follows from
\Cref{lem:anderson-loc-QI} by choosing an arbitrary value $s \in (0,1)$ and  $\delta$ small enough as a function of $D$.
\end{proof}
\subsection{Delocalization}\label{sect:deloc}

\begin{proposition}\label[proposition]{prop:anderson-deloc-QI}
Fix $E\in\mathbb R$ and assume that there exists a constant $c>0$ such that
\[
\liminf_{\eta\rightarrow 0}\mathbb P\!\big(\Im R_{00}(E+\iu\eta)>c\big)>c.
\]
Then there exist $\eta_0(E)>0$ and $C(E)>1$ such that for all $\eta\in(0,\eta_0)$,
\[
\lim_{N\to\infty}\mathbb P\!\left(Q_I\le C\right)=1,
\]
where $I = [E - \eta, E+ \eta]$.
\end{proposition}

\begin{proof}
Let $c_1(E)>0$ be such that \eqref{eq:counting-lower} yields $\lim_{N\to\infty}\mathbb P(|\Lambda_I|\ge c_1 N\eta)=1$. We restrict to the event where $|\Lambda_I|\ge c_1 N\eta$. 
Using \eqref{e:int_step}, we have 
\begin{equation}\label{eq:QI2-upper}
Q_I
=\frac{1}{N}\sum_{j=1}^N P_I(j)^2
\le \frac{4N^2\eta^2}{|\Lambda_I|^2}\cdot \frac{1}{N}\sum_{j=1}^N (\Im G_{jj}(z))^2.
\end{equation}
Using the concentration bound \eqref{eq:concentration-fractional} for $s=2$, we obtain 
\begin{equation} \label{limit:QI deloc}
    \lim_{N\to \infty}\P\left( Q_I\le 4c_1^{-2}\left( \E\Big[\big (\Im G_{11}(z) \big)^2\Big]+\eta \right)\right)= 1.
\end{equation}
From \eqref{limit: diagonal convergence}, together with the deterministic bound $|G_{11}|\le \eta^{-1}$ (from \eqref{e:trivial}), uniform integrability  gives
\begin{equation}\label{limit:L2}
    \lim_{N\to\infty}\E\Big[\big(\Im G_{11}(z) \big)^2\Big]=\E\Big[\big(\Im R_{00}(z)\big)^2\Big].
\end{equation}
Finally, combining \eqref{limit:QI deloc}, \eqref{limit:L2}, and \Cref{l:self-energy} finishes the proof.
\end{proof}

\subsection{Conclusion}
\begin{proof}[Proof of Lemma \ref{lem:tree_resolvent_criterion_anderson}]
    Combining \Cref{prop:anderson-QI-blowup} and \Cref{prop:anderson-deloc-QI} completes the proof.
\end{proof}

\section{Concentration of Resolvent Moments}\label{s:concentration}

In this section, we prove \Cref{l:concentration} using the Gaussian concentration inequality (for the randomness in $V$), a martingale argument based on switching edges in the regular graph (for the randomness in $A$), and a truncation argument.

\subsection{Preliminary Reduction}
For all symmetric $N\times N$ matrices $M$,  vectors $v\in \R^N$, $z \in \bbH$ and $s>0$, define 
\[ 
F_s(z, v , M)=\frac{1}{N}\sum_{i=1}^N \big(\Im S_{ii}(z) \big)^s,
\qquad S = \big( M + \diag(v_1, \dots, v_N) - z\big)^{-1}.
\]
We frequently abbreviate $F_s(z) = F_s(z, v ,M)$, and we write $v_N=(g_1,\dots,g_N)$ and set $\tilde F_s = F_s(z, v_N, -t A_N)$, so that with this choice of arguments, $S(z)=G(z)$. 
We establish the following preliminary concentration estimate in \Cref{s:gauss_conclusion} below, after some preliminary estimates on $F_s(z)$ (and a truncated counterpart). The first exponential term comes from Gaussian concentration in $V_N$, while the second comes from a coarse Azuma--Hoeffding bound for the graph randomness in $A_N$.

\begin{lemma}\label[lemma]{thm:conc-moments}
Fix $z=E+\iu\eta\in\bbH$ and $d\ge 3$. There exists a  constant $C(d) >0$ such that the following hold.

\begin{enumerate}
\item 
For every $s\ge 1$ and every $u>0$,
\begin{equation}\label{eq:conc-sge1}
\P\left(\Big|\tilde F_s(z)-\E \big[ \tilde F_s(z) \big] \Big| \ge u \right)
\le 2\exp\left(-\frac{N\eta^{2s+2}}{8s^2 } u^2\right)
+2\exp\left(-\frac{N\eta^{2s+2}}{C  t^2s^2} u^2\right).
\end{equation}
Further, when $s=1$, the same bound holds for the corresponding averages with $\Im S_{ii}(z)$ replaced by $\Re S_{ii}(z)$. 

\item 
Fix $s \in (0, 1)$. Then for every $u>0$ and every $\delta\in(0,\eta^{-1}]$,
\begin{equation}\label{eq:conc-untrunc-general}
\P\left(\Big|\tilde F_s(z)-\E \big[ \tilde F_s(z) \big] \Big|\ge u+2\delta^s\right)
\le 2\exp \left(-\frac{N\eta^{4}}{8s^2}\delta^{2-2s}\,u^2\right)
+2\exp \left(-\frac{N\eta^{4}}{C  t^2s^2}\delta^{2-2s}\,u^2\right).
\end{equation}
\end{enumerate}
\end{lemma}

\Cref{l:concentration} is an immediate consequence of  \Cref{thm:conc-moments}. 
\begin{proof}[Proof of \Cref{l:concentration}]
We begin with the case $s\ge 1$. This follows from setting $u= \eta$ in \eqref{eq:conc-sge1}. 
Next, for $s\in (0,1)$, if $\eta\ge 1$, then $ 0\le \tilde F_s(z)\le \eta^{-s}\le 1$,
and therefore
\[
\Big|\tilde F_s(z)-\E\big[\tilde F_s(z)\big]\Big|\le 1\le \eta
\]
deterministically, so the claim is trivial. Thus we may assume $\eta\le 1$.
Then the estimate follows from setting $u = \eta/2$ and $\delta = (\eta/4)^{1/s}$ in \eqref{eq:conc-untrunc-general}.
\end{proof}

\subsection{Gaussian Concentration}\label{subsec:conc-setup}

The following Gaussian concentration inequality is standard (see, e.g., \cite[Theorem 5.6]{boucheron2013concentration}). We say a function $f\colon \R^N \rightarrow \R$  is $L$-Lipschitz if $| f(x) - f(y) | \le L \| x - y\|_2$ for all $x,y \in \R^N$.
\begin{lemma}\label[lemma]{l:gaussian_concentration}
Let $X=(X_1, \dots, X_N)$ be a vector of $N$ independent mean zero, variance one Gaussian random variables. Let $f\colon \R^N \rightarrow \R$ be an $L$-Lipschitz function. For all $t >0$, 
\[
\P\left(\Big|
f(X) - \E\big[ f(X) \big] \Big| \ge t 
\right) \le 2\exp\left(
 - \frac{t^2}{2L^2}
\right).
\]
\end{lemma}

We use the notation $\partial_k F_s(z)$ to denote the partial derivative in $v_k$, and we write $\nabla F_s(z)$ for the gradient with respect to $v \in \R^N$. 
A straightforward calculation using the definition of $S$ gives for all $i,k \in [N]$ that 
\begin{equation}\label{eq:res-deriv}
\partial_{k}S_{ii}(z)=-S_{ik}(z)S_{ki}(z).
\end{equation}

We now derive some Lipschitz bounds that are necessary to apply \Cref{l:gaussian_concentration}. These bounds are uniform in the matrix $M$, and hence may be applied conditionally on $A_N$.

\begin{lemma}\label[lemma]{lem:lipschitz-sge1}
For all $s\ge 1$,
\[
\|\nabla F_s(z)\|_2 \le \frac{s}{\eta^{s+1}\sqrt{N}}.
\]
\end{lemma}

\begin{proof}
By the chain rule,
\[
\partial_{k}F_s(z)=\frac{1}{N}\sum_{i=1}^N s \big (\Im S_{ii}(z) \big)^{s-1}\partial_{k}\big (\Im S_{ii}(z) \big).
\]
Using \eqref{eq:res-deriv} and the bound $\Im S_{ii}(z) \le \eta^{-1}$ (analogous to  \eqref{e:trivial}),
\[
\big|\partial_{k}F_s(z)\big|
\le \frac{s}{N}\sum_{i=1}^N (\Im S_{ii})^{s-1}|S_{ik}(z)|^2
\le \frac{s}{N}\cdot \eta^{-(s-1)}\sum_{i=1}^N |S_{ik}(z)|^2.
\]
By the Ward identity \eqref{e:ward} for $S$,
\[
\sum_{i=1}^N |S_{ik}(z)|^2=\frac{\Im S_{kk}(z)}{\eta}\le \frac{1}{\eta^2}.
\]
Therefore,
\[
\big|\partial_{k}F_s(z)\big|\le \frac{s}{N}\cdot \eta^{-(s-1)}\cdot \eta^{-2}
=\frac{s}{N\eta^{s+1}}.
\]
Finally,
\[
\|\nabla F_s(z)\|_2^2=\sum_{k=1}^N \big|\partial_{k}F_s(z)\big|^2
\le N\Big(\frac{s}{N\eta^{s+1}}\Big)^2
=\frac{s^2}{N\eta^{2s+2}},
\]
which implies the claim.
\end{proof}

For every $s\in(0,1)$, the map $x\mapsto x^s$ is not Lipschitz near $x=0$.
We therefore introduce a truncation at level $\delta>0$ and define, for any
$\delta\in(0,\eta^{-1}]$,
\[
F_{s,\delta}(z)=\frac{1}{N}\sum_{i=1}^N \big (\Im S_{ii}(z) \vee\delta \big)^s,
\]
and define $\tilde F_{s,\delta}$ analogously to $\tilde F_{s}$.

\begin{lemma}\label[lemma]{lem:lipschitz-trunc}
Fix $s\in(0,1)$ and $\delta\in(0,\eta^{-1}]$. Then, viewed as a function of
$v\in\R^N$, the map $F_{s,\delta}(z)$ is Lipschitz and
\[
\operatorname{Lip}(F_{s,\delta}(z))\le \frac{s}{\eta^{2}\sqrt{N}}\ \delta^{s-1}.
\]
\end{lemma}

\begin{proof}
Let $\phi_\delta(x)=(x\vee\delta)^s$. Then $\phi_\delta$ is Lipschitz on
$[0,\infty)$, with
\[
\operatorname{Lip}(\phi_\delta)\le s\delta^{s-1}.
\]
Fix $k\in\{1,\dots,N\}$ and $h\in\R$. Then
\[
\begin{aligned}
\big|F_{s,\delta}(v+he_k;z)-F_{s,\delta}(v;z)\big|
&\le \frac{1}{N}\sum_{i=1}^N
\big|\phi_\delta(\Im S_{ii}(v+he_k;z))-\phi_\delta(\Im S_{ii}(v;z))\big| \\
&\le \frac{s\delta^{s-1}}{N}\sum_{i=1}^N
\big|\Im S_{ii}(v+he_k;z)-\Im S_{ii}(v;z)\big|.
\end{aligned}
\]
Since $v\mapsto \Im S_{ii}(v;z)$ is $C^1$, the fundamental theorem of calculus yields
\[
\big|\Im S_{ii}(v+he_k;z)-\Im S_{ii}(v;z)\big|
\le |h|\,\sup_{t\in[0,1]}\big|\partial_k(\Im S_{ii})(v+t h e_k;z)\big|.
\]
Hence
\[
\big|F_{s,\delta}(v+he_k;z)-F_{s,\delta}(v;z)\big|
\le \frac{s\delta^{s-1}|h|}{N}
\sup_{t\in[0,1]}\sum_{i=1}^N
\big|\partial_k(\Im S_{ii})(v+t h e_k;z)\big|.
\]
Note that \eqref{eq:res-deriv} implies $|\partial_{k}(\Im S_{ii})|\le |S_{ik}(z)|^2$.
Then combining this estimate with \eqref{e:trivial} and \eqref{e:ward}, we find
\[
\sum_{i=1}^N \big|\partial_k(\Im S_{ii})(z)\big|
\le \sum_{i=1}^N |S_{ik}(z)|^2
\le \frac{\Im S_{kk}(z)}{\eta}
\le \frac{1}{\eta^2}.
\]
Therefore
\[
\big|F_{s,\delta}(v+he_k;z)-F_{s,\delta}(v;z)\big|
\le \frac{s\delta^{s-1}}{N\eta^2}|h|.
\]
So the Lipschitz constant in each coordinate direction is at most
$\frac{s\delta^{s-1}}{N\eta^2}$.

For arbitrary $v,w\in\R^N$, varying the coordinates one at a time gives
\[
|F_{s,\delta}(v;z)-F_{s,\delta}(w;z)|
\le \frac{s\delta^{s-1}}{N\eta^2}\,\|v-w\|_1
\le \frac{s\delta^{s-1}}{\eta^2\sqrt N}\,\|v-w\|_2.
\]
This proves the Lipschitz bound.
\end{proof}

\subsection{Dependence on the Graph}\label{subsec:graph-conc}

We  use the pairing (configuration) model for random $d$-regular graphs. Let
\[
W=[N]\times[d],
\qquad
m=\frac{Nd}{2}.
\]
A \emph{pairing} of $W$ is a partition of $W$ into unordered sets of cardinality two. 
Given a pairing $\omega$ of $W$, we define the associated multigraph $\mathcal G(\omega)$ on the vertex set $[N]$ as follows. Each element $(i,a)\in W$ is called a half-edge attached to the vertex $i$. For every pair $
\{(i,a),(j,b)\}\in\omega$ 
we place one edge between the vertices $i$ and $j$. Thus, if $i\neq j$, the pair $\{(i,a),(j,b)\}$ contributes one edge joining $i$ and $j$, while if $i=j$, the pair $
\{(i,a),(i,b)\}$ 
contributes one loop at the vertex $i$. Distinct pairs of $\omega$ may project to the same unordered pair $\{i,j\}$, in which case this produces multiple edges.

We write $A(\omega)$ for the  adjacency matrix of $\mathcal G(\omega)$; if a loop
occurs, it contributes $2$ to the corresponding diagonal entry, so that every row sum is
equal to $d$. Let $\hat A_N=A(\omega)$ for $\omega$ chosen uniformly from the set of all
pairings of $W$, and let
\[
\mathcal S_N=\{\hat A_N \text{ is simple}\}.
\]

Every simple $d$-regular graph on $[N]$ corresponds to exactly $(d!)^N$ pairings, hence
\begin{equation}\label{eq:pairing-conditioned-simple}
\Law(\hat A_N \mid \mathcal S_N)=\Law(A_N),
\end{equation}
where $A_N$ denotes the adjacency matrix of the uniform simple $d$-regular graph on $[N]$.
Moreover, for fixed $d$, it is standard that
\[
\lim_{N \rightarrow \infty} 
\P(\mathcal S_N) =  \exp\left(\frac{1-d^2}{4}\right)
\]
(see, e.g., \cite[equation~(5)]{wormald1999models}), and therefore there exists a constant
$p_d>0$ such that
\begin{equation}\label{eq:simple-lower-bound}
\P(\mathcal S_N)\ge p_d
\qquad\text{for all }N .
\end{equation}

We also record the Azuma--Hoeffding martingale concentration estimate that we use. See, e.g., \cite[Chapter~2]{boucheron2013concentration}.

\begin{lemma}\label[lemma]{lem:azuma}
Let $(M_k,\mathcal F_k)_{k=0}^m$ be a martingale such that
\[
|M_k-M_{k-1}|\le c_k
\qquad\text{almost  surely for all }k\in[m].
\]
Then for every $u>0$,
\[
\P\Big(|M_m-M_0|\ge u\Big)
\le 2\exp\Big(-\frac{u^2}{2\sum_{k=1}^m c_k^2}\Big).
\]
\end{lemma}
To apply the previous concentration bound, we require the following lemma on edge switchings of $A$. 
\begin{definition}\label[definition]{def:switching-pairing}
Let $\omega$ be a pairing of a finite set of half-edges $W$. Suppose 
$\{a,b\}, \{c,d\}\in \omega$ 
are two distinct pairs. A \emph{switching of the two pairs $\{a,b\}$ and $\{c,d\}$} is the operation of removing these two pairs and replacing them by one of the two alternative pairings of the four half-edges, namely $\{a,c\},\{b,d\}$ or $\{a,d\}, \{b,c\}$. 

\end{definition}
\begin{lemma}\label[lemma]{lem:graph-lipschitz}
Let $A$ and $A'$ be two symmetric adjacency matrices arising from $d$-regular pairings, and suppose that
$A$ and $A'$ differ by a switching of two pairs of half-edges. Let
\[
S=( -tA+\diag(v)-z)^{-1},
\qquad
S'=( -tA'+\diag(v)-z)^{-1}.
\]
Then the following hold.

\begin{enumerate}
\item For every $s\ge 1$,
\[
\big|F_s(z,v,-tA)-F_s(z,v,-tA')\big|
\le \frac{32 s|t|}{N\eta^{s+1}}.
\]

\item For every $s\in(0,1)$ and every $\delta\in(0,\eta^{-1}]$,
\[
\big|F_{s,\delta}(z,v,-tA)-F_{s,\delta}(z,v,-tA')\big|
\le \frac{32 s|t|}{N\eta^{2}}\,\delta^{s-1}.
\]
\end{enumerate}
\end{lemma}

\begin{proof}
Set $\Delta=-t(A'-A)$. Since $A$ and $A'$ differ by one switching of two pairs, the matrix
$A'-A$ is supported on the at most four vertices incident to the switched half-edges. In particular,
\[
\rank(\Delta)\le 4.
\]
Moreover, $A'-A$ has at most eight nonzero entries, each of magnitude at most $2$
(since loops contribute $2$ to the diagonal). Hence every row and every column of $A'-A$
has $\ell^1$-norm at most $8$, so by the Schur test,
\[
\|\Delta\|\le 8|t|.
\]
By the resolvent identity,
\[
S'-S=- S\Delta S'.
\]
Hence
\[
\rank(S'-S)\le 4,
\qquad
\|S'-S\|\le \|S\|\,\|\Delta\|\,\|S'\|\le \frac{8 |t|}{\eta^2}.
\]
Therefore
\[
\frac{1}{N}\sum_{i=1}^N \big|S'_{ii}-S_{ii}\big|
\le \frac{1}{N}\|S'-S\|_*
\le \frac{4}{N}\|S'-S\|
\le \frac{32 |t|}{N\eta^2},
\]
where $\|\cdot\|_*$ denotes the trace norm.

For $s\ge 1$, the map $x\mapsto x^s$ is $s\eta^{-(s-1)}$-Lipschitz on $[0,\eta^{-1}]$, so
\[
\big|F_s(z,v,-tA)-F_s(z,v,-tA')\big|
\le \frac{s}{N}\eta^{-(s-1)}\sum_{i=1}^N \big|\Im S'_{ii}-\Im S_{ii}\big|
\le \frac{32 s|t|}{N\eta^{s+1}}.
\]
This proves the first claim.

Next, let $\phi_\delta(x) = (x \vee \delta)^s$. Since $s \in (0,1)$, the map
$x \mapsto x^s$ is concave on $(0,\infty)$, and therefore $\phi_\delta$ is
Lipschitz on $[0,\infty)$ with
\[
\operatorname{Lip}(\phi_\delta) \le s\delta^{s-1}.
\]
Since $\Im S_{ii}, \Im S'_{ii} \ge 0$, it follows that
\[
\begin{aligned}
\big|F_{s,\delta}(z,v,-tA)-F_{s,\delta}(z,v,-tA')\big|
&\le \frac{1}{N}\sum_{i=1}^N
\big|\phi_\delta(\Im S'_{ii})-\phi_\delta(\Im S_{ii})\big| \\
&\le \frac{s\delta^{s-1}}{N}\sum_{i=1}^N
\big|\Im S'_{ii}-\Im S_{ii}\big| \\
&\le \frac{32s|t|}{N\eta^2}\delta^{s-1}.
\end{aligned}
\]
This proves the second claim.
\end{proof}

The following is the graph concentration estimate used in the proof of
\Cref{thm:conc-moments}.

\begin{lemma}\label[lemma]{lem:graph-concentration}
Fix $d\ge 3$. There exists a constant $C(d)>0$ such that the following holds.
Let $f$ be a real-valued function of the pairing-model adjacency matrix $A$ satisfying
\[
\big|f(A)-f(A')\big|\le L
\]
whenever $A$ and $A'$ arise from two pairings that differ by one switching of two pairs.
Then for every $u>0$,
\[
\P\Big(\big|f(A_N)-\E\big[f(A_N)\big]\big|\ge u\Big)
\le 2\exp\left(-\frac{u^2}{C NL^2}\right).
\]
\end{lemma}

\begin{proof}
We first prove the corresponding estimate for the configuration-model adjacency matrix
$\hat A_N$.

Choose the random pairing sequentially as follows. Order the half-edges in $W$ lexicographically. At step $k$, let $a_k$ be the smallest
currently unmatched half-edge, and choose its partner $Y_k$ uniformly from the remaining
unmatched half-edges other than $a_k$. This produces a uniformly random pairing. Let
\[
\mathcal F_k=\sigma(Y_1,\dots,Y_k),
\qquad
M_k=\E\big[f(\hat A_N)\mid \mathcal F_k\big],
\qquad
k=0,1,\dots,m.
\]
Then $(M_k,\mathcal F_k)$ is the Doob martingale associated with $f(\hat A_N)$.

We claim that
\begin{equation}\label{eq:martingale-increment-bound}
|M_k-M_{k-1}|\le L
\qquad\text{a.s.\ for all }k\in[m].
\end{equation}
Fix $k$ and condition on $\mathcal F_{k-1}$. Let $U$ denote the set of unmatched half-edges at
time $k-1$, and write $a=a_k$. For each admissible choice $y\in U\setminus\{a\}$, define
\[
g(y)=\E\big[f(\hat A_N)\mid \mathcal F_{k-1},\,Y_k=y\big].
\]
We show that for any two admissible choices $y,y'\in U\setminus\{a\}$,
\begin{equation}\label{eq:g-diameter}
|g(y)-g(y')|\le L.
\end{equation}

If $y=y'$, there is nothing to prove. Assume henceforth that $y\neq y'$.
Since we condition on $\mathcal F_{k-1}$, the unmatched set $U$ and the half-edge $a=a_k$
are deterministic. Under the conditional law given $\mathcal F_{k-1}$ and $Y_k=y$, the remaining
pairs form a uniform random pairing of $U\setminus\{a,y\}$. Thus
\[
g(y)=\E_R\Big[f\big(A(\{\{a,y\}\}\cup R)\big)\Big],
\]
where $R$ is uniform over pairings of $U\setminus\{a,y\}$.
Likewise,
\[
g(y')=\E_{R'}\Big[f\big(A(\{\{a,y'\}\}\cup R')\big)\Big],
\]
where $R'$ is uniform over pairings of $U\setminus\{a,y'\}$.

Now let $R$ be a uniformly random pairing of $U\setminus\{a,y\}$. Let $x$ denote
the partner of $y'$ in $R$, and define $\Phi(R)$ to be the pairing of
$U\setminus\{a,y'\}$ obtained from $R$ by replacing the pair $\{y',x\}$ with $\{y,x\}$ and
leaving all other pairs unchanged. The map $R\mapsto \Phi(R)$ is plainly a bijection from the
set of pairings of $U\setminus\{a,y\}$ to the set of pairings of $U\setminus\{a,y'\}$.
Moreover, the two full pairings $
\{\{a,y\}\}\cup R$ and 
$\{\{a,y'\}\}\cup \Phi(R)$ 
differ by exactly one switching of two pairs. Therefore, by the assumed Lipschitz bound on
$f$,
\[
\Big|f\big(A(\{\{a,y\}\}\cup R)\big)-f\big(A(\{\{a,y'\}\}\cup \Phi(R))\big)\Big|\le L.
\]
Averaging over $R$ yields \eqref{eq:g-diameter}.

Since $M_k=g(Y_k)$ and $M_{k-1}=\E[g(Y_k)\mid \mathcal F_{k-1}]$, the latter being an average
of the values $g(y)$ over admissible $y$, \eqref{eq:g-diameter} implies
\eqref{eq:martingale-increment-bound}. Hence \Cref{lem:azuma} gives
\begin{equation}\label{eq:pairing-conc}
\P\Big(\big|f(\hat A_N)-\E[f(\hat A_N)]\big|\ge u\Big)
\le 2\exp\Big(-\frac{u^2}{NdL^2}\Big),
\end{equation}
since $m=Nd/2$. 
We also record the variance bound
\begin{equation}\label{eq:pairing-var}
\Var\big(f(\hat A_N)\big)
=
\sum_{k=1}^m \E\big[(M_k-M_{k-1})^2\big]
\le mL^2
\le \frac{NdL^2}{2}.
\end{equation}

We now pass from the pairing model to the uniform simple graph. Let
\[
\mu=\E[f(\hat A_N)],
\qquad
\mu_{\mathrm s}=\E\big[f(\hat A_N)\mid \mathcal S_N\big].
\]
Since
\[
\P(\mathcal S_N)\big(\mu_{\mathrm s}-\mu\big)
=
\E\Big[\big(f(\hat A_N)-\mu\big)\mathbf{1}_{\mathcal S_N}\Big],
\]
the Cauchy--Schwarz inequality,  \eqref{eq:simple-lower-bound}, and \eqref{eq:pairing-var} imply
\begin{equation}\label{eq:mean-shift}
|\mu_{\mathrm s}-\mu|
\le \frac{\sqrt{\Var\big(f(\hat A_N)\big)}}{\sqrt{\P(\mathcal S_N)}}
\le C_1(d)\sqrt{Nd}\,L
\end{equation}
for some constant $C_1(d)<\infty$.

Next, using \eqref{eq:pairing-conditioned-simple}, we have for every $u>0$,
\[
\P\Big(\big|f(A_N)-\E\big[f(A_N)\big]\big|\ge u\Big)
=
\P\Big(\big|f(\hat A_N)-\mu_{\mathrm s}\big|\ge u \,\Big|\, \mathcal S_N\Big).
\]
If $u>2C_1(d)\sqrt{Nd}\,L$, then by \eqref{eq:mean-shift},
\[
\Big\{\big|f(\hat A_N)-\mu_{\mathrm s}\big|\ge u\Big\}
\subset
\Big\{\big|f(\hat A_N)-\mu\big|\ge \frac{u}{2}\Big\},
\]
and therefore
\[
\P\Big(\big|f(A_N)-\E[f(A_N)]\big|\ge u\Big)
\le \frac{1}{\P(\mathcal S_N)}
\P\Big(\big|f(\hat A_N)-\mu\big|\ge \frac{u}{2}\Big).
\]
Using \eqref{eq:simple-lower-bound} and \eqref{eq:pairing-conc}, we obtain
\[
\P\Big(\big|f(A_N)-\E[f(A_N)]\big|\ge u\Big)
\le
2p_d^{-1}\exp\Big(-\frac{u^2}{4NdL^2}\Big).
\]
Since $p_d>0$ depends only on $d$, this is bounded by
\[
2\exp\Big(-\frac{u^2}{C_2(d)NdL^2}\Big)
\]
for some constant $C_2(d)<\infty$.

Finally, if $u\le 2C_1(d)\sqrt{Nd}\,L$, then the claimed bound is trivial after enlarging the
constant once more, since the right-hand side can be made at least $1$ on this range of $u$.
This proves the claim.
\end{proof}

\subsection{Conclusion}\label{s:gauss_conclusion}
We begin by deriving two preliminary estimates for the truncated function $\tilde F_{s,\delta}$. 
\begin{lemma}\label[lemma]{lem:conc-trunc}
Fix  $s\in(0,1)$ and $\delta\in(0,\eta^{-1}]$. There exists $C(d)>0$ such that for every $u>0$,
\begin{equation}\label{eq:conc-trunc}
\P\Big(\big|\tilde F_{s,\delta}(z)-\E \big[\tilde F_{s,\delta}(z)\big] \big|\ge u\Big)
\le 2\exp\Big(-\frac{N\eta^{4}}{8s^2}\delta^{2-2s}\,u^2\Big)
+2\exp\Big(-\frac{N\eta^{4}}{Ct^2s^2}\delta^{2-2s}\,u^2\Big).
\end{equation}
\end{lemma}

\begin{proof}
Conditioning on $A_N$ and using \Cref{l:gaussian_concentration,lem:lipschitz-trunc}, we obtain
\[
\P\Big(\big|\tilde F_{s,\delta}(z)-\E \big[\tilde F_{s,\delta}(z)\mid A_N\big] \big|\ge u/2 \,\Big|\, A_N\Big)
\le 2\exp\Big(-\frac{N\eta^{4}}{8s^2}\delta^{2-2s}\,u^2\Big).
\]
Averaging over $A_N$ yields
\[
\P\Big(\big|\tilde F_{s,\delta}(z)-\E \big[\tilde F_{s,\delta}(z)\mid A_N\big] \big|\ge u/2 \Big)
\le 2\exp\Big(-\frac{N\eta^{4}}{8s^2}\delta^{2-2s}\,u^2\Big).
\]

Now define
\[
g_{s,\delta}(A)=\E_V\big[F_{s,\delta}(z,v_N,-tA)\big],
\]
where $\E_V$ denotes expectation in the Gaussian potential only. By \Cref{lem:graph-lipschitz}, $g_{s,\delta}$ is $L$-Lipschitz under one simple switching with
\[
L=\frac{32s|t|}{N\eta^2}\delta^{s-1}.
\]
Hence \Cref{lem:graph-concentration} implies
\[
\P\Big(\big|\E \big[\tilde F_{s,\delta}(z)\mid A_N\big]-\E \big[\tilde F_{s,\delta}(z)\big] \big|\ge u/2\Big)
\le
2\exp\Big(-\frac{N\eta^{4}}{Cdt^2s^2}\delta^{2-2s}\,u^2\Big).
\]
Combining the last two estimates by the triangle inequality yields \eqref{eq:conc-trunc}.
\end{proof}

\begin{lemma}\label[lemma]{lem:conc-untrunc}
Let $s\in(0,1)$ and $\delta\in(0,\eta^{-1}]$. Then for every $u>0$,
\[
\P\Big(\big|\tilde F_s(z)-\E \big[\tilde F_s(z)\big]\big|\ge u+2\delta^s\Big)
\le
\P\Big(\big|\tilde F_{s,\delta}(z)-\E \big[ \tilde F_{s,\delta}(z)\big] \big|\ge u\Big).
\]
\end{lemma}

\begin{proof}
For any $x\ge 0$, we have $(x\vee\delta)^s\ge x^s$, hence $F_{s,\delta}\ge F_s$.
Moreover, if $x\ge\delta$ then $(x\vee\delta)^s-x^s=0$, while if $x<\delta$ then
$(x\vee\delta)^s-x^s=\delta^s-x^s\le \delta^s$. Therefore for all $x\ge 0$,
\[
0\le (x\vee\delta)^s-x^s\le \delta^s.
\]
Averaging over the $S_{ii}$ entries yields $0\le \tilde F_{s,\delta}-\tilde F_s\le \delta^s$, and taking expectations gives 
\[
0\le \E \big[\tilde F_{s,\delta}(z)\big]-\E \big[\tilde F_s(z)\big]\le \delta^s.
\]

Finally, by the triangle inequality,
\[
\big|\tilde F_s-\E [\tilde F_s]\big|
\le \big |\tilde F_{s,\delta}-\E [\tilde F_{s,\delta}]\big| + |\tilde F_{s,\delta}-\tilde F_s| + \big|\E [\tilde F_{s,\delta}]-\E [\tilde F_s]\big|
\le \big|\tilde F_{s,\delta}-\E[ \tilde F_{s,\delta}]\big| + 2\delta^s.
\]
Hence the event $\{|\tilde F_s-\E [\tilde F_s]|\ge u+2\delta^s\}$ is contained in
$\{|\tilde F_{s,\delta}-\E [\tilde F_{s,\delta}]|\ge u\}$, proving the claim.
\end{proof}

We now prove \Cref{thm:conc-moments}.
\begin{proof}[Proof of \Cref{thm:conc-moments}]
We begin with the case $s\ge 1$. Conditioning on $A_N$ and using \Cref{l:gaussian_concentration,lem:lipschitz-sge1}, we obtain
\[
\P\Big(\big|\tilde F_s(z)-\E \big[\tilde F_s(z)\mid A_N\big] \big|\ge u/2\,\Big|\, A_N\Big)
\le 2\exp\Big(-\frac{N\eta^{2s+2}}{8s^2}u^2\Big).
\]
Averaging over $A_N$ yields
\[
\P\Big(\big|\tilde F_s(z)-\E \big[\tilde F_s(z)\mid A_N\big] \big|\ge u/2\Big)
\le 2\exp\Big(-\frac{N\eta^{2s+2}}{8s^2}u^2\Big).
\]

Now define
\[
g_s(A)=\E_V\big[F_s(z,v_N,-tA)\big].
\]
By \Cref{lem:graph-lipschitz}, $g_s$ is $L$-Lipschitz under one simple switching with
\[
L=\frac{32s|t|}{N\eta^{s+1}}.
\]
Hence \Cref{lem:graph-concentration} implies
\[
\P\Big(\big|\E \big[\tilde F_s(z)\mid A_N\big]-\E \big[\tilde F_s(z)\big] \big|\ge u/2\Big)
\le
2\exp\Big(-\frac{N\eta^{2s+2}}{Cdt^2s^2}u^2\Big).
\]
Combining the last two bounds by the triangle inequality proves \eqref{eq:conc-sge1}. The claim about the corresponding averages involving $\Re S_{ii}(z)$ is completely analogous.

Next, by \Cref{lem:conc-untrunc} and \Cref{lem:conc-trunc}, for any $s\in(0,1)$, $\delta\in(0,\eta^{-1}]$, and $u>0$,
\[
\P\Big(\big|\tilde F_s(z)-\E \tilde F_s(z)\big|\ge u+2\delta^s\Big)
\le
\P\Big(\big|\tilde F_{s,\delta}(z)-\E \tilde F_{s,\delta}(z)\big|\ge u\Big)
\]
\[
\le
2\exp\Big(-\frac{N\eta^{4}}{8s^2}\delta^{2-2s}\,u^2\Big)
+
2\exp\Big(-\frac{N\eta^{4}}{Ct^2s^2}\delta^{2-2s}\,u^2\Big),
\]
which is \eqref{eq:conc-untrunc-general}. 
\end{proof}

\section{Diagonal Resolvent Convergence}
\label{s:diagonal_convergence}

This section proves \Cref{thm:diag-resolvent-fixed-z}, which concerns convergence of the diagonal resolvent entries of $W_N$ to $R_{00}$. The two main technical ingredients of the proof are (i) a Combes--Thomas-type exponential decay bound for resolvent entries at fixed $\eta>0$ and (ii) the fact that a random $d$-regular graph is locally tree-like with high probability. Using these inputs, we prove \Cref{thm:diag-resolvent-fixed-z} by truncating the $d$-regular graph around a fixed vertex and then controlling the differences of the resolvent entries of the original and truncated graph by the Combes--Thomas bound. 

\subsection{Preliminary Bounds}\label{subsec:CT}

Let $\mathcal G=(\mathcal V, \mathcal E)$ be any (finite or infinite) graph with maximum degree at most $d$. Here $\mathcal V$ denotes the vertex set and $\mathcal E$ denotes the edge set. 
Let $A$ denote the adjacency operator on $\ell^2(\mathcal V)$, let $V$ also denote a real diagonal potential
(not necessarily bounded), fix $t>0$, and define $H=-tA+V$. Since this notation conflicts with the shorthand introduced in \Cref{s:bethelattice} for the Anderson model on the Bethe lattice, we will always use the more precise notation $H_{d,t} = - t A^{(d)}  + V^{(d)}$ when discussing the Bethe lattice in this section.

For all $X\subset \mathcal V$, write $\chi_X$ for the orthogonal projection  onto $\ell^2(X)$.
For all $X,Y\subset \mathcal V$ define the set-distance
\[
\dist(X,Y)=\inf\{\dist(x,y):x\in X,\ y\in Y\}\in\{0,1,2,\dots\}\cup\{\infty\}.
\]
We record the following straightforward observation, omitting the proof. 
\begin{lemma}\label[lemma]{lem:adj-norm}
If $\mathcal G$ is a graph with maximum degree at most $d$, then $\|A\|\le d$ on $\ell^2(\mathcal V)$.
\end{lemma}

The following Combes--Thomas bound is a key technical input for this section. Its proof is given in \Cref{app:CT-rigorous}.

\begin{lemma}\label[lemma]{lem:CT}
Let $\mathcal G$ be a graph with maximum degree at most $d$. 
Fix $z=E+\iu\eta \in \bbH$ and set
\[
\mu_{\mathrm{CT}}=\log\Big(1+\frac{\eta}{4td}\Big).
\]
Then for any subsets $X,Y\subset \mathcal V$,
\[
\big\|\chi_X(H-z)^{-1}\chi_Y\big\|
\le \frac{2}{\eta} \exp \big( -\mu_{\mathrm{CT}}\dist(X,Y) \big).
\]
\end{lemma}

\subsection{Decoupling and Local Tree-Likeness}\label{subsec:trunc-tree}

We next define a decoupling of the Hamiltonian and quantify its effect on the root resolvent.
Let $\mathcal G=(\mathcal V, \mathcal E)$ be as above and fix a vertex $o\in \mathcal V$. For every $r\in\N$ define the ball
\[
B_r(o)=\{x\in \mathcal V:\dist(o,x)\le r\}.
\]
Define the \emph{decoupling at radius $r$} of the Anderson model on $\mathcal G$ by considering the Anderson model on the graph formed by 
removing all edges between $B_r(o)$ and its complement $B_r(o)^c$, yielding an operator
\[
H^{(r)}=H_{B_r(o)}\oplus H_{B_r(o)^c}
\]
on $\ell^2(\mathcal V)$.
For this section only, we always use $G_N$ when referring to the resolvent of $W_N$; we let $G(z)=(H-z)^{-1}$ and $G^{(r)}(z)=(H^{(r)}-z)^{-1}$ and write $G_{oo}(z)=\langle\delta_o,G(z)\delta_o\rangle$,
and similarly for $G^{(r)}_{oo}(z)$. 

\begin{lemma}\label[lemma]{lem:decouple-stability}
Let $\mathcal G$ be a graph with maximum degree at most $d$. 
With $\mu_{\mathrm{CT}}$ as in \Cref{lem:CT}, for all $r\in\N$,
\[
\big|G_{oo}(z)-G^{(r)}_{oo}(z)\big|\le \frac{4td}{\eta^2}\exp(-2\mu_{\mathrm{CT}}r).
\]
\end{lemma}

\begin{proof}
Let $K=H-H^{(r)}$. Then $K$ is supported on edges crossing the cut between $B_r(o)$ and $B_r(o)^c$, and it equals $-t$ times the adjacency matrix of the graph formed by the cut edges. Hence, by \Cref{lem:adj-norm},
$\|K\|\le t\|A\|\le td$.
It follows from the definitions of $G$ and $G^{(r)}$ that 
\[
G-G^{(r)}=-GKG^{(r)}.
\]
Let $\partial B_r(o)$ denote the set of vertices in $\mathcal V$ incident to at least one cut edge (both inside and outside $B_r(o)$). 
Then $\dist(\{o\},\partial B_r(o))=r$ and $K=\chi_{\partial B_r(o)}K\chi_{\partial B_r(o)}$.
Therefore
\[
|G_{oo}-G^{(r)}_{oo}|
=\big |\langle\delta_o,GKG^{(r)}\delta_o\rangle\big|
\le \|\chi_{\{o\}}G\chi_{\partial B_r(o)}\|\,\|K\|\,
\|\chi_{\partial B_r(o)}G^{(r)}\chi_{\{o\}}\|.
\]
By \Cref{lem:CT} applied to both $H$ and $H^{(r)}$, the first and last factors on the right-hand side are bounded by $\frac{2}{\eta}\exp( -\mu_{\mathrm{CT}}r)$.
Hence
\[
|G_{oo}-G^{(r)}_{oo}|
\le \Big(\frac{2}{\eta}e^{-\mu_{\mathrm{CT}}r}\Big)^2\cdot td
= \frac{4td}{\eta^2}e^{-2\mu_{\mathrm{CT}}r},
\]
as desired.
\end{proof}

We now use local tree-likeness of random $d$-regular graphs to compare finite neighborhoods with the tree.
For a rooted graph $(\mathcal G,o)$ we write $B_r(\mathcal G,o)$ for the rooted radius-$r$ ball.

\begin{lemma}\label[lemma]{lem:local-tree}
Fix $r\in\N$ and let $\mathcal G_N$ be a uniformly random simple $d$-regular graph on $N$ vertices labeled $1$ to $N$.
Then there exists a constant $C(d)>0$ such that for all $N \ge C$, 
\[
\mathbb{P}\big(B_r(\mathcal G_N,1)\text{ contains a cycle}\big)\le \frac{C (d-1)^{2r}}{N}.
\]
\end{lemma}

\begin{proof}
We briefly sketch the standard configuration-model argument.

Let $G^{\mathrm{conf}}_{N,d}$ be the random $d$-regular multigraph from the configuration model (which places $d$ half-edges on each of the $N$ vertices, then forms a perfect matching of the half-edges uniformly at random), and let $\mathbb Q$ be the induced measure on graph events.
We let $C(d)>0$ denote a $d$-dependent constant that may increase at each appearance in this proof.
For fixed $d$, we have $\mathbb{Q}(G^{\mathrm{conf}}_{N,d}\text{is simple})> C^{-1}$  for all $N \ge C$ by \eqref{eq:simple-lower-bound}. 
Therefore, by the definition of conditional probability, we have   for any event $\mathcal{E}$ measurable with respect to the graph and $N \ge C$ that 
\begin{equation}\label{e:transfer}
\mathbb{P} (\mathcal{E}) \le C  \mathbb{Q}(\mathcal{E}).
\end{equation}

We can construct $B_r(1)$ by breadth-first search in the associated configuration model, pairing one half-edge at each step. Let $|B_r(1)|$ denote the number of vertices in $B_r(1)$, and let  $M$ be the total number of half-edges  in $B_r(1)$. 
Deterministically,
\[
\big |B_r(1)\big |\le 1+d\sum_{k=0}^{r-1}(d-1)^k \le C (d-1)^r,\]
which implies
\[
M\le d \big |B_r(1)\big|\le C (d-1)^r,
\]
where we increase $C$ at each step if necessary. 
During the breadth-first pairing procedure, at any step, the probability that a newly exposed half-edge pairs
with one of the at most $M$ half-edges attached to already discovered vertices is bounded by  $C M/ Nd $.
A cycle in $B_r(1)$ can occur only if some such collision happens. By a union bound over the at most $M$ steps,
\[
\mathbb{Q}\big(B_r(1)\text{ contains a cycle in }G^{\mathrm{conf}}_{N,d}\big)
\le M\cdot \frac{C M}{Nd }
\le \frac{C (d-1)^{2r}}{N}.
\]
This estimate then transfers to the uniform simple $d$-regular model using \eqref{e:transfer}.
\end{proof}

\subsection{Proof of Diagonal Resolvent Convergence}\label{subsec:diag-proof}

We now construct the coupling between $d$-regular graph potentials and tree potentials used in the diagonal resolvent comparison.
Sample the tree potential $(V_u)_{u\in \bbV}$ as i.i.d.\ $N(0,1)$ random variables. 
Independently sample the graphs $(\mathcal G_N)_{N\ge 1}$, each uniform over simple $d$-regular graphs on $[N]$.
Let $r_N = \lfloor \log \log N \rfloor$, and define the event 
\[
\mathcal{T}_N=\{B_{r_N}(\mathcal G_N,1)\text{ is a tree}\}.
\]
Recall the shorthand $\bbT = \bbT_d$. On $\mathcal{T}_N$, choose a rooted graph isomorphism
\[
\varphi_N\colon B_{r_N}(\bbT,0)\to B_{r_N}(\mathcal G_N,1).
\]
For concreteness, one may select the isomorphism induced by the breadth-first exploration of the half-edges in the configuration model construction of  $\mathcal G_N$. We define a potential $V^{(N)}$  on  $\mathcal G_N$ in two parts. In the ball, we set
\[
V^{(N)}_x=V_{\varphi_N^{-1}(x)}, \qquad x\in B_{r_N}(\mathcal G_N,1).
\]
Outside the ball, meaning for indices $x\notin B_{r_N}(\mathcal G_N,1)$, we sample the $V^{(N)}_x$ variables as i.i.d.\ $N(0,1)$, which are additionally independent of all other variables in this construction.
On the complement of $\mathcal T_N$, we sample all $V^{(N)}_x$ variables as i.i.d.\ $N(0,1)$, independently from the tree potential.

Let
\[
H_N=-tA_N  + V^{(N)}
\]
and note that $H_N$ has the same distribution as $W_N$. Hence, we will identify $H_N$ with $W_N$ for the remainder of this section.

Let $H_N^{(r)}$ be the decoupling of $H_N$ at radius $r$ around the root $1\in[N]$,
and let $G_N^{(r)}(z)=(H_N^{(r)}-z)^{-1}$.
Define $H^{(r)}$ and $R^{(r)}(z)$ similarly on $\bbT$ (decoupling at radius $r$ around the root $0$). The following lemma is an immediate consequence of these definitions.

\begin{lemma}\label[lemma]{lem:trunc-id}
On the event $\mathcal{T}_N$ we have
\[
G^{(r_N)}_{N,11}(z)=R^{(r_N)}_{00}(z)
\]
under the above coupling.
\end{lemma}

\begin{proof}
On $\mathcal{T}_N$, the rooted balls $B_{r_N}(\mathcal G_N,1)$ and $B_{r_N}(\bbT,0)$ are isomorphic via $\varphi_N$,
and by construction the potentials match under this isomorphism.
The decoupled operators $H_N^{(r_N)}$ and $H^{(r_N)}$ are direct sums of the inside-ball operator and
the outside-ball operator. The diagonal resolvent entry at the root depends only on the inside-ball part.
Since the inside-ball operators are identical under the isomorphism, the diagonal resolvent entries coincide.
\end{proof}
We can now prove \Cref{thm:diag-resolvent-fixed-z}.
\begin{proof}[Proof of \Cref{thm:diag-resolvent-fixed-z}]
By \Cref{lem:decouple-stability} applied to $H_N$ (root $1$) and to $H_{d,t}$ (root $0$) with $r=r_N$, we obtain 
\[
\big|G_{N,11}(z)-G^{(r_N)}_{N,11}(z)\big|\le \frac{4td}{\eta^2}e^{-2\mu_{\mathrm{CT}}r_N},
\qquad
\big|R_{00}(z)-R^{(r_N)}_{00}(z)\big|\le \frac{4td}{\eta^2}e^{-2\mu_{\mathrm{CT}}r_N}.
\]
On the event $\mathcal{T}_N$, \Cref{lem:trunc-id} gives $G^{(r_N)}_{N,11}(z)=R^{(r_N)}_{00}(z)$, hence
\[
\big|G_{N,11}(z)-R_{00}(z)\big|
\le \big|G_{N,11}-G^{(r_N)}_{N,11}\big| + \big|G^{(r_N)}_{N,11}-R^{(r_N)}_{00}\big| + \big|R^{(r_N)}_{00}-R_{00}\big|
\le \frac{8td}{\eta^2}e^{-2\mu_{\mathrm{CT}}r_N}.
\]
Since $r_N\to\infty$, the deterministic bound on the right tends to $0$. Let $\epsilon >0$ be a parameter. Then there exists $C_0(\eps)>0$ such that for $N > C_0$,

\[
\mathbb{P}\Big(\big|G_{N,11}(z)-R_{00}(z)\big|>\varepsilon\Big)\le \mathbb{P}(\mathcal{T}_N^c).
\]
Finally, by \Cref{lem:local-tree}, 
\[
\lim_{N\rightarrow \infty} \mathbb{P}(\mathcal{T}_N^c) = 0 .
\]
This proves $G_{N,11}(z)\to R_{00}(z)$ in probability under the coupling.
\end{proof}

\section{Bounds for the Eigenvalue Counting Function} \label{sect: Global Law}

This section is dedicated to the proof of \Cref{l:spectral_density_bounds}. The main input is the
convergence of the Stieltjes transform of $W_N$ to the Stieltjes transform of the density of states
measure of $H_{d,t}$ on the Bethe lattice. In the Gaussian setting, this follows directly from the
diagonal resolvent convergence proved in \Cref{thm:diag-resolvent-fixed-z}, together with the
$s=1$ case of \Cref{thm:conc-moments}.
Finally, since convergence of Stieltjes transforms implies convergence of the associated spectral measures, the desired lower bound on the eigenvalue counting function follows from a lower bound on the integrated density of states of $H_{d,t}$, which we show in \Cref{app:DOS-bounds}, and the upper bound follows from Wegner's estimate.

\subsection{Notation}
We define 
\[
m_N(z) = \frac{1}{N} \tr G_N(z),
\]
which is the Stieltjes transform of the empirical eigenvalue distribution of $W_N$. 

Next, because $H_{d,t}$ is self-adjoint, the spectral theorem implies the existence of a random projection-valued measure $\mu$ such that $H_{d,t} = \int \lambda \, d\mu (\lambda)$. 
The spectral measure at the root is given by  
\[
\mu_0(S) = \langle \delta_0, \mu(S) \delta_0 \rangle 
\]
for all Borel sets $S \subset \R$, and the density of states measure $\nu$ is given by 
\[
\nu(S) = \E\big[ \mu_0(S)  \big].
\]
For all $z\in \bbH$, the Stieltjes transform of $\nu$ is
\[
m_{\bbT}(z) = \int_\R \frac{1}{\lambda  -z } \, d\nu(\lambda).
\]

\subsection{Preliminaries}\label{subsec:geisinger-prelim}

We first state an upper bound  on the integrated density of states of $H_{d,t}$, which is a consequence of a well-known estimate of Wegner \cite[(5.27)]{aizenman2013resonant}. 

\begin{lemma}\label[lemma]{l:wegner}
For every $d \ge 3$ and $t>0$, there exists a constant $C(d,t) > 0$ such that for all intervals $I \subset \R$, we have 
$\nu (I) \le C | I |$.
\end{lemma}

We also require a lower bound that complements the previous Wegner estimate. It is proved in \Cref{app:DOS-bounds}.

\begin{lemma}\label[lemma]{l:dos_lower}
For every $d \ge 3$, $t>0$, and $D>0$, there exists a constant $c(d,t,D) > 0$ such that for all intervals $I \subset [-D,D]$, we have 
$ \nu (I) \ge c | I |$.
\end{lemma}

We next show convergence of Stieltjes transforms.

\begin{proposition}\label[proposition]{thm:global-law}
Fix $z\in \bbH$. Then $\lim_{N\rightarrow \infty} m_N(z) = m_{\bbT}(z)$ in probability.
\end{proposition}

\begin{proof}
Write $z = E+\iu\eta$, where $\eta>0$ is fixed. By exchangeability of the vertices,
\begin{equation}\label{e:exchange}
\E \big[m_N(z)\big] =\frac1N\sum_{j=1}^N \E\big[ G_{jj}(z)\big]=\E \big[G_{11}(z)\big].
\end{equation}
By \Cref{thm:diag-resolvent-fixed-z}, there exists a coupling such that $G_{11}(z)\to R_{00}(z)$
in probability. Since $|G_{11}(z)|\le \eta^{-1}$ and $|R_{00}(z)|\le \eta^{-1}$ deterministically,
the bounded convergence theorem yields, after using \eqref{e:exchange}, 
\begin{equation}\label{eq:annealed-global-law-short}
\lim_{N\to\infty} \E \big[m_N(z)\big] = \lim_{N\to\infty} \E \big[G_{11}(z)\big]=\E \big[R_{00}(z)\big]=m_{\bbT}(z).
\end{equation}

It remains to control the fluctuations of $m_N(z)$. By the $s=1$ case of \Cref{thm:conc-moments}
applied to the imaginary parts, and to the corresponding averages of the real parts, there exists
$C(d)>0$ such that for all $u>0$ and all sufficiently large $N$,
\[
\P\Big(\big|m_N(z)-\E\big[ m_N(z)\big] \big|\ge u\Big)
\le 2C\exp\big(-N\eta^4u^2/(4C)\big),
\]
which implies $m_N(z)-\E [m_N(z)]\to 0$ 
in probability. Combining this with \eqref{eq:annealed-global-law-short} completes the proof.
\end{proof}

\subsection{Conclusion}\label{subsec:counting-proof}

\begin{proof}[Proof of \Cref{l:spectral_density_bounds}]
Let 
\[\hat \mu_N = \frac{1}{N} \sum_{i=1}^N \delta_{\lambda_i}
\]
denote the empirical measure of $W_N$. By a basic property of Stieltjes transforms, the convergence of Stieltjes transforms in probability at every $z\in \bbH$ provided by \Cref{thm:global-law} implies that $\hat \mu_N$ converges weakly to $\nu$ in probability (see, e.g., \cite[Lemma~2.1]{benaych2016lectures}). Since $\nu$ is absolutely continuous, by \Cref{l:wegner}, we have that $\hat \mu_N(I)$ converges to $\nu(I)$ in probability for every interval $I \subset \R$. 
The conclusion now follows from \Cref{l:wegner} and \Cref{l:dos_lower}.
\end{proof}

\section{Uniform Moment Bound for the Resolvent}
\label{s:negativemoment}

In this section, we prove \Cref{l:self-energy}.

\subsection{Preliminaries}

It is well known (see, e.g., \cite[(2.2)]{aggarwal2025mobility}) that for every $z \in \bbH$, the resolvent entry $R'_{00}(z)$ for the Anderson model on a variant of $\bbT_d$, where the root degree is $d-1$ and the degree of all other vertices is $d$, solves the recursive distribution equation 
\begin{equation}\label{eq:RDE-cavity-rev}
\Gamma_\eta \stackrel{d}{=} \frac{1}{V_0-E-\iu\eta-t^2\sum_{j=1}^{d-1}\Gamma_{\eta,j}},
\end{equation}
where $V_0\sim N(0,1)$, the variables $\Gamma_{\eta,1},\dots,\Gamma_{\eta,d-1}$ and $\Gamma_\eta$ are i.i.d.\ and all variables on the right-hand side are independent.

We write
\begin{equation}\label{xetayeta}
\Gamma_\eta=X_\eta+\iu Y_\eta,
\end{equation}
where $X_\eta$ and $Y_\eta$ are real. We have $Y_\eta>0$ almost surely and
\begin{equation}\label{eq:XY-formulas-rev}
X_\eta=\frac{U_\eta}{U_\eta^2+S_\eta^2},\qquad
Y_\eta=\frac{S_\eta}{U_\eta^2+S_\eta^2},
\end{equation}
where
\begin{equation}\label{eq:US-def-rev}
U_\eta=V_0-E-t^2\sum_{j=1}^{d-1} X_{\eta,j},
\qquad
S_\eta=\eta+t^2\sum_{j=1}^{d-1} Y_{\eta,j},
\end{equation}
with the $X_{\eta,j}$ and $Y_{\eta,j}$ defined as in \eqref{xetayeta}. 
Moreover, the Schur complement formula (see, e.g., \cite{aizenman2013resonant}) implies 
\begin{equation}\label{eq:root-formula-rev}
R_{00}(E+\iu\eta)=\frac{1}{V_0-E-\iu\eta-t^2\sum_{j=1}^{d}\Gamma_{\eta}^{(j)}},
\end{equation}
where $V_0\sim N(0,1)$ is independent of the i.i.d.\ random variables
$\Gamma_{\eta}^{(1)},\dots,\Gamma_{\eta}^{(d)}$, each distributed as $\Gamma_\eta$.
In particular, defining 
\begin{equation}\label{eq:B-def-rev}
B_\eta=\eta+t^2\sum_{j=1}^{d}Y_{\eta}^{(j)},
\end{equation}
where $Y_{\eta}^{(j)}$ is defined analogously to $Y_{\eta, j}$,  there is a real random variable $A_\eta$, independent of $V_0$, such that
\begin{equation}\label{eq:ImR-root-rev}
\Im R_{00}(E+\iu\eta)=\frac{B_\eta}{(V_0-A_\eta)^2+B_\eta^2}.
\end{equation}

For all lemmas in this section, we retain the notation and hypothesis of \Cref{l:self-energy}. In particular, $E \in \R$, $d \ge 10$, and $t>0$ are always fixed. We also let $\rho$ denote the density of a mean zero, variance one Gaussian random variable. The constants in this section may depend on the constant in the hypothesis of \Cref{l:self-energy}, but we always omit this dependence in the notation.

\begin{lemma}\label[lemma]{l:uniform-seed}
There exist constants $\eta_0(d,t,E)\in(0,1]$, $y_*(d,t,E)>0$, and $p_*(d,t,E)>0$,
such that for all $\eta \in (0, \eta_0]$,
\begin{equation}\label{eq:uniform-seed-rev}
\P(Y_\eta>y_*)\ge p_*.
\end{equation}
\end{lemma}

\begin{proof}
Choose $p_0>0$ and $\eta_1\in(0,1]$ so that, for all $0<\eta\le \eta_1$,
\begin{equation}\label{eq:seed-root-prob-rev}
\P\bigl(\Im R_{00}(E+\iu\eta)>c\bigr)\ge 2p_0.
\end{equation}
Fix $0<\eta\le \eta_1$. Conditional on $(A_\eta,B_\eta)$ in \eqref{eq:ImR-root-rev}, the event
$\{\Im R_{00}(E+\iu\eta)>c\}$ is empty unless $B_\eta<1/c$, in which case it is equivalent to
\[
|V_0-A_\eta|<\sqrt{\frac{B_\eta}{c}-B_\eta^2}.
\]
Since $\|\rho\|_\infty=(2\pi)^{-1/2}$,
\begin{equation}\label{eq:conditional-root-prob-rev}
\P\bigl(\Im R_{00}(E+\iu\eta)>c\mid A_\eta,B_\eta\bigr)
\le 2\|\rho\|_\infty c^{-1/2} B_\eta^{1/2}\one_{\{B_\eta<1/c\}}.
\end{equation}
Set
\[
C_{\mathrm G}=2\|\rho\|_\infty c^{-1/2}.
\]
Taking expectations and using \eqref{eq:seed-root-prob-rev},
\begin{equation}\label{eq:sqrtB-lower-rev}
2p_0\le C_{\mathrm G}\E\bigl[B_\eta^{1/2}\one_{\{B_\eta<1/c\}}\bigr].
\end{equation}
Choose
\[
b_*=\min\left\{\frac{1}{4c}, \left(\frac{p_0}{C_{\mathrm G}}\right)^2\right\}.
\]
Then
\[
\E\bigl[B_\eta^{1/2}\one_{\{B_\eta<1/c\}}\bigr]
\le b_*^{1/2}+c^{-1/2}\P(B_\eta>b_*),
\]
and therefore \eqref{eq:sqrtB-lower-rev} implies
\[
2p_0\le p_0+C_{\mathrm G}c^{-1/2}\P(B_\eta>b_*).
\]
Hence, for all $0<\eta\le \eta_1$,
\begin{equation}\label{eq:B-lower-tail-seed-rev}
\P(B_\eta>b_*)\ge p_1>0, \qquad p_1=\frac{p_0 c^{1/2}}{C_{\mathrm G}}.
\end{equation}

Now set
\[
\eta_0=\min\{\eta_1,b_*/2\},
\qquad
y_*=\frac{b_*}{2 d t^2}.
\]
If $0<\eta\le \eta_0$ and $B_\eta>b_*$, then
\[
t^2\sum_{j=1}^{d}Y_\eta^{(j)}\ge B_\eta-\eta>b_*/2,
\]
so for at least one index $j$ we have $Y_\eta^{(j)}>y_*$. By the union bound and symmetry,
\[
p_1\le \P\left(\bigcup_{j=1}^{d}\{Y_\eta^{(j)}>y_*\}\right)
\le d \P(Y_\eta>y_*).
\]
Thus \eqref{eq:uniform-seed-rev} holds with
\[
p_*=\frac{p_1}{d}>0.
\qedhere
\]
\end{proof}

\begin{lemma}\label[lemma]{l:onestep-tails}
 There exists a constant $C(d,t,E)>0$, such that for every $0<\eta\le \eta_0$, the following bounds hold.
 \begin{enumerate}
 \item For all $u >0$,
 \begin{equation}
 \label{eq:X-tail-uncond-rev}
 \P(|X_\eta|>u) \le \frac{C}{u}, \qquad \P(Y_\eta>u)\le \frac{C}{u}.
 \end{equation}
 \item
 For all $u \ge 1$, 
 \begin{equation}
 \label{eq:U-tail-rev}
 \P(|U_\eta|>u)\le \frac{C}{u}. 
 \end{equation}
 \item For all $u\ge 2\eta_0$,
 \begin{equation}
 \label{eq:S-tail-rev}
 \P(S_\eta>u) \le \frac{C}{u}. 
 \end{equation}
 \end{enumerate}
\end{lemma}

\begin{proof}
For all $s>0$ and $u>0$, the set
\[
\left\{x\in\R: \frac{|x|}{x^2+s^2}>u\right\}
\]
has Lebesgue measure at most $2/u$. Therefore, conditioning on $S_\eta$ and 
\[
A=E+t^2\sum_{j=1}^{d-1} X_{\eta,j}, 
\]
we obtain from \eqref{eq:XY-formulas-rev} that
\[
\P\bigl(|X_\eta|>u\mid A,S_\eta\bigr)\le \frac{2\|\rho\|_\infty}{u}.
\]
Taking expectations yields the first bound of \eqref{eq:X-tail-uncond-rev}.

Likewise, the set
\[
\left\{x\in\R: \frac{s}{x^2+s^2}>u\right\}
\]
has Lebesgue measure at most $1/u$, hence
\[
\P\bigl(Y_\eta>u\mid A,S_\eta\bigr)\le \frac{\|\rho\|_\infty}{u},
\]
and therefore the second bound of \eqref{eq:X-tail-uncond-rev} follows after taking expectation.

For \eqref{eq:U-tail-rev}, we use \eqref{eq:US-def-rev} and the union bound to get 
\[
\P\big(|U_\eta|>u\big)
\le \P\big(|V_0-E|>u/2\big)+\sum_{j=1}^{d-1}\P\left(|X_{\eta,j}|>\frac{u}{2(d-1)t^2}\right).
\]
The first term is bounded by $C/u$ for $u\ge 1$ and the second by \eqref{eq:X-tail-uncond-rev}, which gives \eqref{eq:U-tail-rev}.

Finally, if $u\ge 2\eta_0$ and $0<\eta\le \eta_0$, then
\[
S_\eta=\eta+t^2\sum_{j=1}^{d-1} Y_{\eta,j}>u\]
implies
\[
\sum_{j=1}^{d-1} Y_{\eta,j}>\frac{u}{2t^2},
\]
so at least one index satisfies $Y_{\eta,j}>u/(2(d-1)t^2)$. Hence
\[
\P(S_\eta>u)
\le \sum_{j=1}^{d-1} \P\left(Y_{\eta,j}>\frac{u}{2(d-1)t^2}\right),
\]
and \eqref{eq:S-tail-rev} follows from \eqref{eq:X-tail-uncond-rev}.
\end{proof}

\subsection{Bootstrap and Conclusion} 

\begin{lemma}\label[lemma]{l:recursive-lower-tail}
 For $\varepsilon>0$, define
\[
F_\eta(\varepsilon)=\P(Y_\eta\le \varepsilon).
\]
There exist constants $C_1(d,t,E),C_2(d,t,E) > 0 $ and $\varepsilon_1(d,t,E)\in(0,1]$ such that, for all $0<\varepsilon\le \varepsilon_1$ and $0<\eta\le \eta_0$,
\begin{equation}\label{eq:F-recursion-rev}
F_\eta(\varepsilon)
\le F_\eta(C_1\varepsilon^{1/3})^{d-1} + C_2\varepsilon^{1/3}.
\end{equation}
\end{lemma}

\begin{proof}
We use that the event $\{Y_\eta\le \varepsilon\}$ is contained in 
\[
\{S_\eta\le \varepsilon^{1/3}\}
\cup
\{\varepsilon^{1/3}<S_\eta\le \varepsilon^{-1/3}, Y_\eta\le \varepsilon\}
\cup
\{S_\eta>\varepsilon^{-1/3}\}.
\]
On the first event, each summand in \eqref{eq:US-def-rev} satisfies
\[
Y_{\eta,j}\le t^{-2}\varepsilon^{1/3},
\]
so by independence,
\begin{equation}\label{eq:T1-rev}
\P(S_\eta\le \varepsilon^{1/3})\le F_\eta(t^{-2}\varepsilon^{1/3})^{d-1}.
\end{equation}
On the third event, \eqref{eq:S-tail-rev} yields
\begin{equation}\label{eq:T3-rev}
\P(S_\eta>\varepsilon^{-1/3})\le C\varepsilon^{1/3}
\end{equation}
for all $\varepsilon\le \varepsilon_1$ small enough that $\varepsilon^{-1/3}\ge 2\eta_0$.

Finally, on the middle event,
\[
\varepsilon\ge Y_\eta=\frac{S_\eta}{U_\eta^2+S_\eta^2},
\]
which implies
\[
U_\eta^2\ge \frac{S_\eta}{\varepsilon}-S_\eta^2.
\]
For fixed $\varepsilon$, the function $s\mapsto s/\varepsilon-s^2$ is increasing on
$[\varepsilon^{1/3},\varepsilon^{-1/3}]$ once $\varepsilon$ is small enough, because its derivative is
$\varepsilon^{-1}-2s>0$ there. Hence, for $0<\varepsilon\le \varepsilon_1$ sufficiently small,
\[
\frac{S_\eta}{\varepsilon}-S_\eta^2
\ge \frac{\varepsilon^{1/3}}{\varepsilon}-\varepsilon^{2/3}
=\varepsilon^{-2/3}-\varepsilon^{2/3}
\ge \frac12\varepsilon^{-2/3}.
\]
Therefore,
\[
\P\bigl(\varepsilon^{1/3}<S_\eta\le \varepsilon^{-1/3}, Y_\eta\le \varepsilon\bigr)
\le \P\bigl(|U_\eta|\ge 2^{-1/2}\varepsilon^{-1/3}\bigr)
\le C\varepsilon^{1/3}
\]
by \eqref{eq:U-tail-rev}, for all $\varepsilon\le \varepsilon_1$ small enough that $2^{-1/2}\varepsilon^{-1/3}\ge 1$.
Set $C_1=\max\{1,t^{-2}\}$. 
Since $F_\eta$ is nondecreasing and $t^{-2}\le C_1$, \eqref{eq:T1-rev} implies
\[
\P(S_\eta\le \varepsilon^{1/3})\le F_\eta(C_1\varepsilon^{1/3})^{d-1}.
\]
Combining this with \eqref{eq:T3-rev} and the bound on the middle event proves
\eqref{eq:F-recursion-rev}.
\end{proof}

\begin{lemma}\label[lemma]{l:poly-lower-tail}
There exist constants $C_*(d,t,E)>0$ and $\varepsilon_*(d,t,E)\in(0,1]$ such that, for all
$0<\varepsilon\le \varepsilon_*$ and $0<\eta\le \eta_0$,
\begin{equation}\label{eq:Y-lower-tail-poly-rev}
\P(Y_\eta\le \varepsilon)\le C_*\varepsilon^\beta,
\qquad
\beta=\frac{31}{301}.
\end{equation}
\end{lemma}

\begin{proof}
Let $C_1,C_2,\varepsilon_1$ be as in \Cref{l:recursive-lower-tail}, and set
\[
\bar q=1-p_*\in(0,1).
\]
Define a deterministic sequence for $k\ge 0$ by
\[
r_0=\bar q,\qquad r_{k+1}=\frac{1}{2} (r_k^{\,d-1}+r_k).
\]
Since $0<r_k<1$, we have $r_{k+1}<r_k$ for all $k$, and hence $r_k\rightarrow 0$ (since any limit point must satisfy the previous recursion). Choose
$m\in \mathbb N$ such that
\[
r_m\le \frac12.
\]

Choose
\[
\tilde\varepsilon_0\in (0,\min\{y_*,\varepsilon_1,1\}]
\]
so small that, for the sequence defined by
\[
\tilde\varepsilon_{n+1}=\Bigl(\frac{\tilde\varepsilon_n}{C_1}\Bigr)^3,
\qquad
L_n=-\ln \tilde\varepsilon_n,
\]
the following hold:
\begin{equation}\label{eq:short-smallness}
L_0\ge 300\ln C_1,
\qquad
L_1\ge \max\{50\ln(2C_2),\,2\},
\end{equation}
and
\begin{equation}\label{eq:short-bootstrap}
C_2\tilde\varepsilon_{k+1}^{1/3}\le \frac12\bigl(r_k-r_k^{\,d-1}\bigr)
\qquad\text{for all }0\le k<m.
\end{equation}
Since $L_{n+1}=3L_n+3\ln C_1$, the first condition implies
\begin{equation}\label{eq:L-ratio-short}
L_{n+1}\le \frac{301}{100}L_n
\qquad\text{for all }n\ge 0.
\end{equation}

Define
\[
q_n=\sup_{0<\eta\le \eta_0}F_\eta(\tilde\varepsilon_n),
\qquad
\nu_n=-\ln q_n.
\]
Because $\tilde\varepsilon_0\le y_*$, \Cref{l:uniform-seed} gives
\[
q_0\le \bar q=r_0.
\]
Also, since $C_1\tilde\varepsilon_{n+1}^{1/3}=\tilde\varepsilon_n$, \Cref{l:recursive-lower-tail} yields
\begin{equation}\label{e:reback}
q_{n+1}\le q_n^{\,d-1}+C_2\tilde\varepsilon_{n+1}^{1/3}.
\end{equation}

We claim that
\[
q_k\le r_k\qquad\text{for all }0\le k\le m.
\]
This is true for $k=0$. If $k<m$ and $q_k\le r_k$, then by \eqref{eq:short-bootstrap},
\[
q_{k+1}\le q_k^{\,d-1}+C_2\tilde\varepsilon_{k+1}^{1/3}
\le r_k^{\,d-1}+\frac12\bigl(r_k-r_k^{\,d-1}\bigr)
= \frac12\bigl(r_k^{\,d-1}+r_k\bigr)
= r_{k+1}.
\]
Thus the claim follows by induction. In particular,
\[
q_m\le r_m\le \frac12,
\]
and hence
\begin{equation}\label{eq:nu-m-short}
\nu_m\ge \ln 2.
\end{equation}

From \eqref{e:reback} and \(a+b\le 2\max\{a,b\}\), and the definitions of $L_{n}$ and $\nu_n$, we get
\[
\nu_{n+1}\ge \min\Bigl\{(d-1)\nu_n,\frac13L_{n+1}-\ln C_2\Bigr\}-\ln 2.
\]
Since $L_{n+1}\ge L_1\ge 50\ln(2C_2)$, we obtain
\[
\frac13L_{n+1}-\ln C_2-\ln 2
=
\frac13L_{n+1}-\ln(2C_2)
\ge \frac{31}{100}L_{n+1}.
\]
Therefore
\begin{equation}\label{eq:nu-rec-short}
\nu_{n+1}\ge \min\Bigl\{(d-1)\nu_n-\ln 2,\frac{31}{100}L_{n+1}\Bigr\}
\qquad\text{for all }n\ge 0.
\end{equation}

We claim that there exists $n_*\ge m$ such that
\begin{equation}\label{eq:first-hit-short}
\nu_{n_*}\ge \frac{31}{301}L_{n_*+1}.
\end{equation}
Suppose not. Then
\[
\nu_n<\frac{31}{301}L_{n+1}\qquad\text{for every }n\ge m.
\]
If for some $n\ge m$ we had
\[
(d-1)\nu_n-\ln 2\ge \frac{31}{100}L_{n+1},
\]
then \eqref{eq:nu-rec-short} would imply
\[
\nu_{n+1}\ge \frac{31}{100}L_{n+1}\ge \frac{31}{301}L_{n+2},
\]
by \eqref{eq:L-ratio-short}, contradicting the assumption. Hence necessarily
\[
(d-1)\nu_n-\ln 2<\frac{31}{100}L_{n+1}
\qquad\text{for all }n\ge m,
\]
and then \eqref{eq:nu-rec-short} gives
\[
\nu_{n+1}\ge (d-1)\nu_n-\ln 2
\qquad\text{for all }n\ge m.
\]
Iterating,
\[
\nu_n\ge (d-1)^{n-m}\Bigl(\nu_m-\frac{\ln 2}{d-2}\Bigr)+\frac{\ln 2}{d-2}
\qquad\text{for all }n\ge m.
\]
By \eqref{eq:nu-m-short} and the assumption $d\ge 10$, the coefficient
\[
\nu_m-\frac{\ln 2}{d-2}
\]
is strictly positive. On the other hand, \eqref{eq:L-ratio-short} implies
\[
L_{n+1}\le \Bigl(\frac{301}{100}\Bigr)^{n-m}L_{m+1}.
\]
Since $d-1\ge 9>301/100$, the inequality
\[
(d-1)^{n-m}\Bigl(\nu_m-\frac{\ln 2}{d-2}\Bigr)+\frac{\ln 2}{d-2}
<
\frac{31}{301}\Bigl(\frac{301}{100}\Bigr)^{n-m}L_{m+1}
\]
cannot hold for all large $n$, contradiction. This proves \eqref{eq:first-hit-short}.

Now \eqref{eq:first-hit-short} propagates. Indeed, if for some $n\ge n_*$,
\[
\nu_n\ge \frac{31}{301}L_{n+1},
\]
then, because $d-1\ge 9$ and $L_{n+1}\ge L_1\ge 2$,
\[
(d-1)\nu_n-\ln 2
\ge \frac{279}{301}L_{n+1}-\ln 2
\ge \frac{31}{100}L_{n+1}.
\]
Thus \eqref{eq:nu-rec-short} implies
\[
\nu_{n+1}\ge \frac{31}{100}L_{n+1}\ge \frac{31}{301}L_{n+2},
\]
again by \eqref{eq:L-ratio-short}. By induction,
\[
\nu_n\ge \frac{31}{301}L_{n+1}\qquad\text{for all }n\ge n_*.
\]
Equivalently,
\[
q_n\le \tilde\varepsilon_{n+1}^{31/301}\qquad\text{for all }n\ge n_*.
\]

Finally, let $\varepsilon_*=\tilde\varepsilon_{n_*}$, and fix $0<\varepsilon\le \varepsilon_*$. Choose
$n\ge n_*$ such that
\[
\tilde\varepsilon_{n+1}<\varepsilon\le \tilde\varepsilon_n.
\]
Since $F_\eta$ is nondecreasing,
\[
F_\eta(\varepsilon)\le q_n\le \tilde\varepsilon_{n+1}^{31/301}<\varepsilon^{31/301}.
\]
This proves \eqref{eq:Y-lower-tail-poly-rev} with $C_*=1$.
\end{proof}

\begin{lemma}\label[lemma]{l:root-self-energy}
 There exists a constant $C(t,d,E)>0$ such that
\begin{equation}\label{eq:Binv-bound-rev}
\sup_{0<\eta\le \eta_0}\E[B_\eta^{-1}]\le C.
\end{equation}
\end{lemma}

\begin{proof}
Let $\beta=31/301$ and $\varepsilon_*>0$ be given by \Cref{l:poly-lower-tail}. If
$0<\varepsilon\le t^2\varepsilon_*$ and $B_\eta\le \varepsilon$, then each summand in \eqref{eq:B-def-rev} with $1\le j\le d$ satisfies
\[
Y_\eta^{(j)}\le \frac{\varepsilon}{t^2},
\]
since all summands are nonnegative. Therefore, by independence and \Cref{l:poly-lower-tail},
\[
\P(B_\eta\le \varepsilon)
\le \prod_{j=1}^{d}\P\left(Y_\eta^{(j)}\le \frac{\varepsilon}{t^2}\right)
\le C\varepsilon^{\beta d }.
\]

Set $s_0=\max\{1,(t^2\varepsilon_*)^{-1}\}$. 
For all $s\ge s_0$, we have $s^{-1}\le t^2\varepsilon_*$, hence
\[
\P(B_\eta<s^{-1})\le C s^{-\beta d}.
\]
Therefore,
\[
\E[B_\eta^{-1}]
\le 1+\int_1^\infty \P(B_\eta<s^{-1})\,ds
\le 1+\int_1^{s_0}1\,ds + C\int_{s_0}^\infty s^{-\beta d}\,ds.
\]
Since $\beta d>1$, the right-hand side is finite uniformly in $0<\eta\le \eta_0$.
This proves \eqref{eq:Binv-bound-rev}.
\end{proof}

We now return to \Cref{l:self-energy}.

\begin{proof}[Proof of \Cref{l:self-energy}]
Let $\eta\in(0,\eta_0]$. Conditioning on $(A_\eta,B_\eta)$ in \eqref{eq:ImR-root-rev},
\[
\E\left[\big(\Im R_{00}(E+\iu\eta)\bigr)^2 \mid A_\eta,B_\eta\right]
=\int_{\R}\rho(v)\frac{B_\eta^2}{((v-A_\eta)^2+B_\eta^2)^2}\,dv.
\]
Since $\rho$ is bounded,
\[
\E\left[\bigl(\Im R_{00}(E+\iu\eta)\bigr)^2 \mid A_\eta,B_\eta\right]
\le \|\rho\|_\infty\int_{\R}\frac{B_\eta^2}{(x^2+B_\eta^2)^2}\,dx
=\frac{\pi\|\rho\|_\infty}{2B_\eta}.
\]
Taking expectations and applying \Cref{l:root-self-energy},
\begin{equation}\label{eq:small-eta-second-moment-rev}
\sup_{0<\eta\le \eta_0}\E\left[\bigl(\Im R_{00}(E+\iu\eta)\bigr)^2\right]<\infty.
\end{equation}
For $\eta\in[\eta_0,1]$, the trivial resolvent bound $|R_{00}(E+\iu\eta)|\le \eta_0^{-1}$ gives
\[
\sup_{\eta\in[\eta_0,1]}\E\left[\bigl(\Im R_{00}(E+\iu\eta)\bigr)^2\right]\le \eta_0^{-2}<\infty.
\]
Together with \eqref{eq:small-eta-second-moment-rev}, this proves \Cref{l:self-energy}.
\end{proof}

\appendix

\section{Combes--Thomas Bound}\label[appendix]{app:CT-rigorous}

In the proof of  Combes--Thomas estimates, one typically conjugates the Hamiltonian $H$ 
by an exponential weight operator
\[
(W\psi)(x)=e^{\mu w_X(x)}\psi(x),\qquad w_X(x)=\dist(x,X),
\]
 then writes $H_\mu=W^{-1}HW$ and manipulates resolvents. On an infinite graph, $W$ is usually an unbounded
operator on $\ell^2(\mathcal V)$, so one may worry about domain issues. In this appendix, we address these concerns 
in a standard way. We first prove the desired Combes--Thomas estimate on finite volumes, then pass to the infinite graph by strong resolvent convergence and a finite-rank sandwiching argument.

Throughout this appendix, we  let $\mathcal G=(\mathcal V, \mathcal E)$ be a graph (finite or infinite) with maximum degree at most $d$,  and we retain the notation introduced in \Cref{s:diagonal_convergence}. Additionally, for vertices $v,w$ in a graph, we write $v \sim w$ if there is an edge between $v$ and $w$. 

\subsection{Finite-Volume Estimate}\label{app:CT-finite}

Fix any finite set $X\subset \mathcal V$ and  $r\in\N$, and define the (finite) neighborhood
\[
\Lambda_r=\{x\in\mathcal  V:\dist(x,X)\le r\}.
\]
If $\mathcal G$ has degree bounded by $d$ and $X$ is finite, then $\Lambda_r$ is finite for every $r$.
Define the  restriction $H_{\Lambda_r}$ as the operator on $\ell^2(\Lambda_r)$ obtained by removing all
edges between $\Lambda_r$ and $\Lambda_r^c$, and write 
\[
H_{\Lambda_r}=-tA_{\Lambda_r}+V_{\Lambda_r}.
\]

\begin{lemma}\label[lemma]{lem:CT-finite}
Fix $z=E+\iu\eta\in \bbH$ and  $t>0$. Recall $\mu_{\mathrm{CT}}$ from \Cref{lem:CT}. 
Then for any finite sets $X,Y\subset \Lambda_r$,
\[
\big\|\chi_X\,(H_{\Lambda_r}-z)^{-1}\,\chi_Y\big\|
\le \frac{2}{\eta}\,e^{-\mu_{\mathrm{CT}} \dist(X,Y)}.
\]
\end{lemma}

\begin{proof}
We abbreviate $\mu = \mu_{\mathrm{CT}}$. 
Since $\Lambda_r$ is finite, $H_{\Lambda_r}$ is a finite Hermitian matrix and all algebraic manipulations below are permissible. Fix a nonempty set $X\subset \Lambda_r$ and set $w_X(x)=\dist(x,X)$ for $x\in \Lambda_r$. Define
the bounded weight operator $W_r$ on $\ell^2(\Lambda_r)$ by
\[
(W_r\psi)(x)=e^{\mu w_X(x)}\psi(x),\qquad x\in\Lambda_r,
\]
and set
\[
H_{\Lambda_r,\mu}=W_r^{-1}H_{\Lambda_r}W_r=-tW_r^{-1}A_{\Lambda_r}W_r+V_{\Lambda_r}.
\]
In the last equality, we used that since $V_{\Lambda_r}$ is diagonal, it commutes with $W_r$. 

For neighbors $x\sim y$ in $\Lambda_r$ we have
$|w_X(x)-w_X(y)|\le 1$, so
\[
\big|(W_r^{-1}A_{\Lambda_r}W_r)_{xy}\big|=e^{\mu(w_X(y)-w_X(x))}\le e^\mu.
\]
Hence $J_r=W_r^{-1}A_{\Lambda_r}W_r-A_{\Lambda_r}$ is supported on edges and satisfies
$|(J_r)_{xy}|\le e^\mu-1$ for $x\sim y$. Analogously to \Cref{lem:adj-norm},
\[
\|J_r\|\le d(e^\mu-1).
\]
Therefore
\[
\|H_{\Lambda_r,\mu}-H_{\Lambda_r}\|\le td(e^\mu-1).
\]
By the definition of $\mu$, we have $td(e^\mu-1)=\eta/4$. Also,
$\|(H_{\Lambda_r}-z)^{-1}\|\le 1/\eta$. Thus
\[
\|(H_{\Lambda_r,\mu}-H_{\Lambda_r})(H_{\Lambda_r}-z)^{-1}\|\le \frac14,
\]
so $H_{\Lambda_r,\mu}-z$ is invertible, and
\[
\|(H_{\Lambda_r,\mu}-z)^{-1}\|
\le \frac{1}{1-1/4}\,\|(H_{\Lambda_r}-z)^{-1}\|
\le \frac{2}{\eta}.
\]
Now $\chi_XW_r=\chi_X$ and, since $w_X(y)\ge \dist(X,Y)$ for all $y\in Y$,
\[
\|W_r^{-1}\chi_Y\|\le e^{-\mu\dist(X,Y)}.
\]
Hence
\[
\|\chi_X(H_{\Lambda_r}-z)^{-1}\chi_Y\|
=\|\chi_XW_r(H_{\Lambda_r,\mu}-z)^{-1}W_r^{-1}\chi_Y\|
\le \frac{2}{\eta}e^{-\mu\dist(X,Y)}.
\]
This proves the claim.
\end{proof}

\subsection{Passage to Infinite Volume}\label{app:CT-infinite}

We define the decoupled operator 
\[
\widehat H_r=H_{\Lambda_r}\oplus H_{\Lambda_r^c},
\]
on $\ell^2(\mathcal V)$  by cutting all edges between $\Lambda_r$ and $\Lambda_r^c$ in the underlying graph and removing the corresponding entries from $H$ (so that $\widehat H_r$ is self-adjoint).
If $X,Y\subset \Lambda_r$, then 
\begin{equation}\label{eq:sandwich-identity}
\chi_X(\widehat H_r-z)^{-1}\chi_Y=\chi_X(H_{\Lambda_r}-z)^{-1}\chi_Y .
\end{equation}

\begin{lemma}\label[lemma]{lem:sr-conv}
For every $z \in \bbH$,
\[
\lim_{r\rightarrow \infty} (\widehat H_r-z)^{-1} \phi = (H-z)^{-1}\phi
\]
in norm for all $\phi \in \ell^2(\mathcal V)$.
\end{lemma}

\begin{proof}
Let $C_c(\mathcal V)$ be the finitely supported functions on $\mathcal V$. Since $A$ is bounded and $V$ is a self-adjoint
multiplication operator, $C_c(\mathcal V)$ is a core for $H$ and every $\widehat H_r$.
For any $\psi\in C_c(\mathcal V)$, let $\mathcal C_X\subset \mathcal V$ be the union of the connected components of $\mathcal G$ that intersect $X$. Since $\operatorname{supp}\psi\cap \mathcal C_X$ is finite and every vertex in $\mathcal C_X$ has finite distance to $X$, there exists $R$ such that
\[
\operatorname{supp}\psi\cap \mathcal C_X\subset \Lambda_{R-1}.
\]
Hence for all $r\ge R$,
\[
\operatorname{supp}\psi\subset \Lambda_{r-1}\cup \mathcal C_X^c.
\]
Now every edge removed in passing from $H$ to $\widehat H_r$ has one endpoint in $\Lambda_r$ and the other in $\Lambda_r^c$, so both endpoints lie in $\mathcal C_X$; moreover, the endpoint in $\Lambda_r$ actually lies in $\Lambda_r\setminus \Lambda_{r-1}$. Therefore no removed edge is incident to a vertex where $\psi$ is nonzero: on $\mathcal C_X$ the support of $\psi$ lies inside $\Lambda_{r-1}$, while on $\mathcal C_X^c$ the operators $\widehat H_r$ and $H$ already agree. Hence for all $r\ge R$,
\[
\widehat H_r\psi=H\psi,
\]
and therefore $\widehat H_r\psi\to H\psi$ as $r\rightarrow \infty$ for every $\psi\in C_c(\mathcal V)$.
By the core criterion for strong resolvent convergence (e.g., \cite[Theorem VIII.25]{reedsimon1972mmp1}), this implies
$\widehat H_r\to H$ in the strong resolvent sense, which is exactly the desired conclusion.
\end{proof}

\begin{lemma}\label[lemma]{lem:finite-rank}
Let $(T_r)_{r=1}^\infty$ and $T$ be bounded operators on a Hilbert space such that $\sup_r\|T_r\|<\infty$, and suppose $T_r\to T$ in the strong operator topology as $r\rightarrow \infty$.
If $P$ and $Q$ are finite-rank orthogonal projections, then
\[
\lim_{r\rightarrow \infty} 
\|PT_rQ-PTQ\| = 0.
\]
\end{lemma}

\begin{proof}
Let $F=\mathrm{Ran}(Q)$; then $F$ is finite-dimensional. Define $S_r=P(T_r-T)Q$, which maps $F$ into the
finite-dimensional space $\mathrm{Ran}(P)$. Strong convergence implies that for each $v\in F$,
$\|S_r v\|\to 0$. Since $F\cap\{v:\|v\|=1\}$ is compact and $\sup_r\|S_r\|<\infty$, pointwise convergence on the
unit sphere implies uniform convergence,
\[
\sup_{v\in F,\ \|v\|=1}\|S_r v\|\to 0,
\]
which is exactly $\|S_r\|\to 0$.
\end{proof}

We can now prove \Cref{lem:CT}.
\begin{proof}[Proof of \Cref{lem:CT}]
We abbreviate $\mu = \mu_{\mathrm{CT}}$.  
First, we prove the claim when both $X$ and $Y$ are finite.
Fix finite $X,Y\subset \mathcal V$ and choose $r$ large enough so that $X\cup Y_{\mathrm{conn}}\subset \Lambda_r$, where $Y_{\mathrm{conn}}$ is the intersection of $Y$ with the union of the connected components of $\mathcal G$ meeting $X$. 
By \Cref{lem:CT-finite} and the identity \eqref{eq:sandwich-identity},
\begin{equation}\label{e:ct_1}
\big\|\chi_X(\widehat H_r-z)^{-1}\chi_Y\big\|
= \big\|\chi_X(H_{\Lambda_r}-z)^{-1}\chi_{Y_{\mathrm{conn}}}\big\|
\le \frac{2}{\eta}\,e^{-\mu\dist(X,Y)}.
\end{equation}
By \Cref{lem:sr-conv}, $(\widehat H_r-z)^{-1}\to (H-z)^{-1}$ in the strong operator topology. Further, we have the elementary bound 
$\sup_r\|(\widehat H_r-z)^{-1}\|\le 1/\eta$. This allows us to apply \Cref{lem:finite-rank} with
$T_r=(\widehat H_r-z)^{-1}$, $T=(H-z)^{-1}$, $P=\chi_X$, and $Q=\chi_Y$ 
to obtain
\[
\big\|\chi_X(\widehat H_r-z)^{-1}\chi_Y - \chi_X(H-z)^{-1}\chi_Y\big\|\to 0.
\]
Together with \eqref{e:ct_1}, this proves the claim when $X$ and $Y$ are finite. 

For the general case, pick increasing sequences of finite sets $(X_n)_{n=1}^\infty$ and $(Y_n)_{n=1}^\infty$ such that $\lim_{n\rightarrow \infty} X_n = X$ and $\lim_{n\rightarrow \infty} Y_n = Y$. Since $X_n\subset X$ and $Y_n\subset Y$, we have
\[
\dist(X_n,Y_n)\ge \dist(X,Y).
\]
Therefore, by the finite-set Combes--Thomas estimate,
\begin{equation}\label{e:uniform}
\|\chi_{X_n}(H-z)^{-1}\chi_{Y_n}\|
\le \frac{2}{\eta}\,e^{-\mu\dist(X_n,Y_n)}
\le \frac{2}{\eta}\,e^{-\mu\dist(X,Y)}
\qquad\text{for all }n.
\end{equation}
Next, for any $\psi\in \ell^2(\mathcal V)$,
\begin{align*}
\bigl(\chi_{X_n}(H-z)^{-1}\chi_{Y_n}-\chi_X(H-z)^{-1}\chi_Y\bigr)\psi
=&
(\chi_{X_n}-\chi_X)(H-z)^{-1}\chi_{Y_n}\psi\\
&+\chi_X(H-z)^{-1}(\chi_{Y_n}-\chi_Y)\psi.
\end{align*}
Since $(H-z)^{-1}$ is bounded and $\chi_{X_n}\to \chi_X$, $\chi_{Y_n}\to \chi_Y$ strongly,
both terms on the right-hand side converge to zero in norm.
In combination with \eqref{e:uniform}, this shows that for any $\psi \in \ell^2(\mathcal V)$,
\[
\|\chi_X(H-z)^{-1}\chi_Y\psi\|
=
\lim_{n\to\infty}\|\chi_{X_n}(H-z)^{-1}\chi_{Y_n}\psi\|
\le
\frac{2}{\eta}e^{-\mu\dist(X,Y)}\|\psi\|.
\]
Taking the supremum over all $\psi$ with $\|\psi\|=1$ yields
\[
\|\chi_X(H-z)^{-1}\chi_Y\|
 \le
\frac{2}{\eta} e^{-\mu\dist(X,Y)}.
\]
This proves the claim.
\end{proof}

\section{Lower Bound on the Density of States}\label[appendix]{app:DOS-bounds}

This appendix is devoted to the proof of \Cref{l:dos_lower}. 

For each neighbor $y\sim 0$, let $\bbT_y$ denote the connected component of
$\bbT\setminus\{0\}$ containing $y$, and let $H_{\bbT_y}$ be the restriction
of $H_{d,t}$ to $\ell^2(\bbT_y)$. For $z=E+\iu \eta$ with $\eta>0$, define
\[
\Gamma_y(z)=\langle \delta_y,(H_{\bbT_y}-z)^{-1}\delta_y\rangle.
\]
By the Schur complement formula,
\[
R_{00}(z)=\langle \delta_0,(H_{d,t}-z)^{-1}\delta_0\rangle
=\frac{1}{V_0-z-t^2\sum_{y\sim0}\Gamma_y(z)}.
\]

\begin{lemma}\label[lemma]{lem:forward-resolvent-small}
There exists a constant $p_*>0$ such that for all $E\in\R$, all $\eta>0$, and
all neighbors $y\sim 0$,
\[
\P\big(|\Gamma_y(E+\iu \eta)|\le 1\big)\ge p_*.
\]
\end{lemma}

\begin{proof}
Fix $z=E+\iu\eta$ and $y\sim0$. Conditioning on the potentials in
$\bbT_y\setminus\{y\}$, the Schur complement formula gives
\[
\Gamma_y(z)=\frac{1}{V_y-a-\iu b}
\]
for some random variables $a\in\R$ and $b>0$ that are independent of $V_y$.
Hence
\[
\P\big(|\Gamma_y(z)|\le 1 \,\big|\, a,b\big)
=\P\big((V_y-a)^2+b^2\ge 1 \,\big|\, a,b\big)
\ge \P\big(|V_y-a|\ge 1\big).
\]
Since the density of $V_0$ is even and strictly decreasing on $[0,\infty)$, among all
intervals of length $2$ the interval $[-1,1]$ has the largest Gaussian mass.
Therefore
\[
\P(|V_y-a|\ge 1)\ge \P(|V_y|\ge 1)>0
\]
for all $a\in\R$. Setting $p_* = \P(|V_y|\ge 1)$ and averaging over $a,b$ proves the claim.
\end{proof}

\begin{proposition}\label[proposition]{prop:local-dos-lower}
Fix $d \ge 3$, $t>0$, and $D>0$.  Then there exists a
constant $c_0(d,t,D)>0$ such that for all $E\in[-D,D]$ and all
$\eta\in(0,1]$,
\[
\frac{1}{\pi}\Im m_\bbT(E+\iu\eta)\ge c_0.
\]
\end{proposition}

\begin{proof}
Fix $E\in[-D,D]$ and $\eta\in(0,1]$, and set $z=E+\iu\eta$. Define the event
\[
\mathcal G_z=\bigcap_{y\sim0}\big\{|\Gamma_y(z)|\le 1\big\}.
\]
Since the forward subtrees $(\bbT_y)_{y\sim0}$ are disjoint and the potentials
are i.i.d., the random variables $(\Gamma_y(z))_{y\sim0}$ are independent.
Hence, by \Cref{lem:forward-resolvent-small},
\[
\P(\mathcal G_z)\ge p_*^d.
\]

Conditioning on the sigma-field generated by $(V_x)_{x\neq 0}$, define
\[
A_z=E+t^2\sum_{y\sim0}\Re \Gamma_y(z),
\qquad
B_z=\eta+t^2\sum_{y\sim0}\Im \Gamma_y(z).
\]
Then $B_z>0$ and
\[
R_{00}(z)=\frac{1}{V_0-A_z-\iu B_z}.
\]
Therefore,
\[
\frac{1}{\pi}\Im m_{\bbT}(z)
=\frac{1}{\pi}\E\big[\Im R_{00}(z)\big]
=\E\!\left[\int_{\R}\rho(v)\,\frac{1}{\pi}\frac{B_z}{(v-A_z)^2+B_z^2}\,dv\right],
\]
where $\rho$ is the density of $V_0$. 
On the event $\mathcal G_z$ we have
\[
|A_z|
\le D+t^2d,
\qquad
0<B_z\le 1+t^2d.
\]
Set
\[
K_1=D+t^2 d ,
\qquad
K_2=1+t^2 d.
\]
Define, for $a\in\R$ and $b>0$,
\[
F(a,b)=\int_{\R}\rho(v)\,\frac{1}{\pi}\frac{b}{(v-a)^2+b^2}\,dv.
\]
Because $\rho$ is continuous and strictly positive, $F$ extends continuously to
$\R\times[0,\infty)$ by setting $F(a,0)=\rho(a)$ (see, e.g., \cite[Theorem 7.4]{axler2013harmonic}). Hence
\[
\inf_{|a|\le K_1,\; 0\le b\le K_2} F(a,b)>0.
\]
Denote this infimum by $c_1$. 
On $\mathcal G_z$ we have $|A_z|\le K_1$ and $0<B_z\le K_2$, so
\[
F(A_z,B_z)\ge c_1.
\]
Consequently,
\[
\frac{1}{\pi}\Im m_{\bbT}(z)
=\E\big[F(A_z,B_z)\big]
\ge c_1\,\P(\mathcal G_z)
\ge c_1 p_*^d.
\]
Thus the claim holds with $c_0=c_1p_*^d$.
\end{proof}

We are now ready for the proof of \Cref{l:dos_lower}.

\begin{proof}[Proof of \Cref{l:dos_lower}]
By \Cref{l:wegner}, the measure $\nu$ is absolutely continuous with
respect to Lebesgue measure, so there exists $n\in L^\infty(\R)$ such that
\[
\nu(d\lambda)=n(\lambda)\,d\lambda.
\]
For $\eta>0$, let
\[
P_\eta(x)=\frac{1}{\pi}\frac{\eta}{x^2+\eta^2}
\]
denote the Poisson kernel. Since $m_{\bbT}$ is the Stieltjes transform of $\nu$, we
have
\[
\frac{1}{\pi}\Im m_{\bbT}(E+\iu \eta)
=\int_{\R} P_\eta(E-\lambda)\,n(\lambda)\,d\lambda
=(P_\eta*n)(E).
\]
By \Cref{prop:local-dos-lower}, for every $E\in[-D,D]$ and every $\eta\in(0,1]$,
\[
(P_\eta*n)(E)\ge c_0.
\]
Now fix an interval $I\subset[-D,D]$. Integrating over $E\in I$ gives
\[
\int_I (P_\eta*n)(E)\,dE \ge c_0|I|
\qquad\text{for all }\eta\in(0,1].
\]
Since $(P_\eta)_{\eta>0}$ is an approximate identity and $n\in L^\infty(\R)$,
we have
\[
\lim_{\eta\downarrow0}\int_I (P_\eta*n)(E)\,dE
=\int_I n(E)\,dE
=\nu(I).
\]
Therefore, $\nu(I)\ge c_0|I|$, as desired.
\end{proof}

\printbibliography

\end{document}